\documentclass[12pt]{article}
\usepackage{graphicx} 
\usepackage{amsmath}
\usepackage{amssymb}
\usepackage{amsthm}
\usepackage{mathtools}
\usepackage{hyperref}
\newcommand{\UkjT}{U(\boldsymbol{k},j,\mathbb T)}
\newcommand{\LinftyUkjT}{L^\infty(\mathbb B^n)^{\UkjT}}
\newcommand{\Tkqrph}{\mathcal T_{k\text{-qr,ph}}}
\newcommand{\Tkqr}{\mathcal T_{k\text{-qr}}}

\newcommand{\Tkqrcj}{\mathcal T_{k\text{-qr,}c_j}}
\newcommand{\Jmuj}{\mathcal{J}_{\mu,j}}
\newcommand{\LinftydT}{L^{\infty}(\mathbb{B}^d)^{\mathbb{T}}}
\newcommand{\contdT}{\mathcal{C}_{\bB^d}^{\bT}}
\newcommand{\Dkqrcj}{\mathcal{D}_{k\text{-qr},c_j}}
\newcommand{\Dkqrph}{\mathcal{D}_{k\text{-qr,ph}}}

\newcommand{\bB}{\mathbb{B}}
\newcommand{\bN}{\mathbb{N}}
\newcommand{\bT}{\mathbb{T}}
\newcommand{\bC}{\mathbb{C}}

\newcommand{\ka}{\kappa}

\newcommand{\cA}{\mathcal{A}}
\newcommand{\cT}{\mathcal{T}}
\newcommand{\cB}{\mathcal{B}}
\newcommand{\cJ}{\mathcal{J}}
\newcommand{\cP}{\mathcal{P}}

\newcommand{\cL}{\mathcal{L}}
\newcommand{\cK}{\mathcal{K}}
\newcommand{\cD}{\mathcal{D}}

\newcommand{\bdB}{\mathbf{B}}
\newcommand{\bdk}{\boldsymbol{k}}
\newcommand{\bdS}{\boldsymbol{S}}
\newcommand{\bdv}{\boldsymbol{v}}
\newcommand{\bdsig}{\boldsymbol{\sigma}}
\newcommand{\bdT}{\boldsymbol{T}}
\newcommand{\bde}{\boldsymbol{e}}
\newcommand{\bdH}{\boldsymbol{H}}
\newcommand{\bdr}{\boldsymbol{r}}
\newcommand{\bdxi}{\boldsymbol{\xi}}
\newcommand{\bdone}{\boldsymbol{1}}
\newcommand{\bdt}{\boldsymbol{t}}
\newcommand{\bdV}{\boldsymbol{V}}

\newcommand{\pconv}[1]{\operatorname{hull}_{\operatorname{pol}}(#1)}

\newcommand{\Tph}{\mathcal T_{\textup{ph}}}

\newcommand{\LinftyUk}{L^\infty(\bB^n)^{U(\bdk)}}
\newcommand{\LinftyTm}{L^\infty(\bB^n)^{\bT^m}}
\newcommand{\LinftyTmpol}{L^\infty(\bdB^{\bdk})^{\bT^m}}
\newcommand{\LinftyjT}{L^\infty(\bB^{k_j})^{\bT}}

\newcommand{\re}{\operatorname{Re}}

\newcommand{\contTm}{\mathcal{C}^{\bT^m}}

\usepackage[
  backend=biber,
  style=numeric-comp,
  sorting=nyt,
  giveninits=true
]{biblatex}
\AtBeginBibliography{\raggedright}

\newtheorem{thm}{Theorem}[section]
\newtheorem{prop}[thm]{Proposition}
\newtheorem{lem}[thm]{Lemma}
\newtheorem{cor}[thm]{Corollary}

\theoremstyle{definition}
\newtheorem{rem}[thm]{Remark}
\newtheorem{ex}[thm]{Example}

\title{A localization framework for $\bT^m$-Invariant Toeplitz Algebras:\\
Applications to Gelfand Theory}
\author{Miguel Angel Rodriguez Rodriguez}
\date{September 2026}

\begin{document}

\maketitle

\begin{center}
\emph{To the memory of Nikolai Vasilevski}
\end{center}

\begin{abstract}
We study $\mathbb T^m$-invariant Toeplitz operators on weighted Bergman spaces over the unit ball. A natural unitary transformation represents each such operator as a direct sum of restrictions of Toeplitz operators acting on a polyball. Under an $L^\infty$-valued angular continuity assumption, the corresponding symbols extend from $\mathbb N_0^m$ to a norm-continuous family indexed by the maximal ideal space of the $C^*$-algebra generated by $k$-quasi-radial Toeplitz operators. This extension yields a localization framework for Toeplitz operator algebras. As an application, we consider commutative Banach algebras obtained by adjoining Toeplitz operators whose symbols are invariant under mixed unitary and circle actions. We identify their local quotient algebras on finite-coordinate strata and at infinity and use these identifications to describe their maximal ideal spaces and Gelfand transforms explicitly.
\end{abstract}



\section{Introduction}

Let $\lambda>-1$. The weighted Bergman space
$\cA^2_\lambda(\bB^n)$ consists of the holomorphic functions on the unit ball
$\bB^n\subset\bC^n$ that are square integrable with respect to the weighted
Lebesgue measure defined in \eqref{eq:def-vdlambda}. For
$\varphi\in L^\infty(\bB^n)$, the Toeplitz operator with symbol $\varphi$ is
\[
T_\varphi f=P_{\cA^2_\lambda(\bB^n)}(\varphi f),
\qquad f\in\cA^2_\lambda(\bB^n),
\]
where $P_{\cA^2_\lambda(\bB^n)}$ is the Bergman projection.

Symmetries of the underlying domain are often reflected in the structure of
the corresponding Toeplitz operator algebras, i.e., operator algebras generated by Toeplitz operators. In one complex variable,
Grudsky, Quiroga-Barranco and Vasilevski showed that, under suitable hypotheses, certain Toeplitz
$C^*$-algebras on the unit disk are commutative precisely when their symbols
are invariant under a maximal abelian group of M\"obius transformations; see
\cite{GrudskyQuirogaVasilveski2006}.

Higher-dimensional analogues of the one-dimensional result have produced
several families of commutative Toeplitz algebras on the unit ball, including
naturally occurring commutative Banach algebras for which the $C^*$-algebras generated by
the same operators are noncommutative.
A broad class of such Banach algebras is obtained by starting with the
commutative $C^*$-algebra $\Tkqr$ generated by Toeplitz operators with so-called
$k$-quasi-radial symbols (see Section \ref{sec:U(k)-and-Tm-invariant-operators}) and adjoining non-normal Toeplitz operators
$T_{c_j}$, $j=1,\ldots,m$, whose symbols are invariant under suitable mixed
unitary and circle actions. We denote these algebras by
\[
\Tkqrph=\cB(\Tkqr,T_{c_1},\ldots,T_{c_m}),
\]
where $\cB$ denotes the corresponding generated unital Banach algebra.

These Banach algebras have been extensively studied (see
\cite{BauerVasilevski2012,GarciaVasilevski2015,Quiroga_Sanchez_2021,RodriguezVasilevski2021,Bauer_Rodriguez2022}
and the references therein). The largest class of symbols $c_j$ known to
produce such algebras was identified in \cite{Quiroga2021}. Nevertheless,
despite partial results in \cite{RodriguezRodriguezPHD2024}, their maximal ideal
space and Gelfand transform
had not been characterized.

This paper has two main contributions. First, we introduce a localization
framework for $\bT^m$-invariant Toeplitz operators on
$\cA^2_\lambda(\bB^n)$. A natural unitary transformation represents each such
operator as a direct sum of restrictions of Toeplitz operators acting on a
Bergman space over a polyball. The resulting family is initially indexed by
$\bN_0^m$ and, under an $L^\infty$-valued angular continuity assumption,
extends continuously to the maximal ideal space $M(\Tkqr)$. Each
$\mu\in M(\Tkqr)$ then determines a local quotient algebra, so the global
operator algebra can be analyzed through finite-coordinate and asymptotic
localizations. This construction is related in spirit to the local principles
of Allan and Douglas--Varela; see
\cite{Rabinovich2004LimitOperators,Vasilevski2008}.

The framework is one of the principal novelties of the paper. Its
operator decomposition and continuity results do not require commutativity.
Hence, extensions to noncommutative $C^*$-algebras
and other group actions, including variants of the $\bT^m$-actions studied in
\cite{Vasilevski2018}, are natural directions for further work.

Our first main result, Theorem~\ref{thm:Tna=oplus-T-polyball}, gives the
polyball decomposition. Proposition~\ref{prop:continuity-wildeamu} and
Corollaries~\ref{cor:uniform-approx-aka-amu}--\ref{cor:Tmu-to-Tkappaalpha}
establish the continuity of the extended symbols
$\widetilde c_\mu$, $\mu\in M(\Tkqr)$, and their Toeplitz operators under the
stated angular continuity assumption. Theorems~\ref{thm:local-alg-finite}
and~\ref{thm:local-alg-infinite} then identify the local quotients of
\[
\Tkqrcj=\cB(\Tkqr,T_{c_j}).
\]
If the $j$th coordinate of the stratum containing $\mu$ is finite, the local
quotient is isometrically isomorphic to the Banach algebra generated by the
finite-dimensional restriction
\[
T_{k_j,\widehat c_{j,\mu}}\big|_{H_{k_j,\omega_j}}.
\]
If that coordinate is infinite, it is isometrically isomorphic to the algebra
generated in the Calkin algebra by
\[
[T_{k_j,\widehat c_{j,\mu}}]_{\cK}.
\]

As our second main contribution, we use these local identifications to obtain the previously missing
Gelfand theory of $\Tkqrph$. For compact sets $D_j(\mu,c_j)\subset\bC$
determined by finite-dimensional spectra or by the boundary values of
$\widehat c_{j,\mu}$, Theorem~\ref{thm:Gelfand theory Tkqrph} proves that
\[
M(\Tkqrph)
\cong
\bigcup_{\mu\in M(\Tkqr)}
\{\mu\}\times\pconv{D_1(\mu,c_1)}\times\cdots\times
\pconv{D_m(\mu,c_m)},
\]
and gives an explicit formula for the Gelfand transform. This characterization
settles, under the stated continuity assumptions, the problem that had
remained open for this class of algebras.

The paper is organized as follows. Section~\ref{sec:preliminaries} reviews
weighted Bergman spaces, polyballs and spherical-coordinate integration.
Section~\ref{sec:Tm-inv-Toep-op} develops the operator-family and localization
framework. Section~\ref{sec:Operator algebras} identifies the local quotient
algebras and applies them to the maximal ideal space and Gelfand transform of
$\Tkqrph$.

\section{Preliminaries}
\label{sec:preliminaries}

\subsection{Bergman spaces on the unit ball}

The main operators considered in this paper act on the Bergman space $\cA^2_\lambda(\bB^n)$ for a fixed $n\in\bN$. Since we also use Bergman spaces on balls of other dimensions, we begin with the general case $d\in\bN$.

For $d\in \bN$, we denote by $\cA^2_\lambda(\bB^d)$ the weighted Bergman space on the unit ball $\bB^d$ of $\bC^d$ with weight parameter $\lambda>-1$, defined as the space of all holomorphic functions on $\bB^d$ that are square integrable with respect to the measure
\begin{equation}
    \label{eq:def-vdlambda}
    dv_{d,\lambda}(z)=c_{d,\lambda}(1-|z|^2)^\lambda dV_d(z),
\end{equation}
where $V_d$ denotes Lebesgue measure restricted to $\bB^d$, and $c_{d,\lambda}$ is the normalization constant
\[
c_{d,\lambda}=\frac{\Gamma(d+\lambda+1)}{\pi^d \Gamma(\lambda+1)}.
\]
When $\lambda=0$, we simply write $\cA^2(\bB^d)\coloneqq\cA^2_0(\bB^d)$, $c_d\coloneqq c_{d,0}$ and $dv_d\coloneqq c_d dV_d$.
As is well known, $\cA_\lambda^2(\bB^d)$ is a reproducing kernel Hilbert space and a closed subspace of $L^2(\bB^d,dv_{d,\lambda})$. We denote by $P_{\cA^2_\lambda(\bB^d)}$ the orthogonal projection from $L^2(\bB^d,dv_{d,\lambda})$ onto $\cA^2_\lambda(\bB^d)$.

Let $S^{2d-1}$ denote the unit sphere in $\bC^d$ and let $\sigma_d$ denote the standard, unnormalized surface measure on $S^{2d-1}$. Thus, $\sigma_d(S^{2d-1})=\frac{2\pi^d}{(d-1)!}$.

We write $z\in \bB^d$ in spherical coordinates as $z=r\xi$ with $r\in [0,1)$ and $\xi\in S^{2d-1}$.
The spherical-coordinate formula (see, for example, \cite{Zhu2005SpacesHolomorphic}) states that, for $f\in L^1(\bB^d,V_d)$,
\begin{equation}
    \label{eq:measures-polar-coordinates}
    \int_{\bB^d}f(z)dV_d(z)=
\int_0^1 r^{2d-1}dr
\int_{S^{2d-1}}f(r\xi)d\sigma_d(\xi).
\end{equation}
For $\alpha,\beta\in \bN_0^d$, one also has
\begin{equation}
    \label{eq:int-z-alpha-conj-z-beta-dvlambda}
    \int_{\bB^d}z^\alpha \overline{z^\beta}dv_{d,\lambda}(z)
=
\delta_{\alpha,\beta}
\frac{\alpha!\Gamma(d+\lambda+1)}{\Gamma(d+|\alpha|+\lambda+1)},
\end{equation}
where we use the multi-index notation
\[
\alpha!=\alpha_1!\cdots\alpha_d!,\qquad
|\alpha|=\alpha_1+\cdots+\alpha_d,
\qquad
z^\alpha=z_1^{\alpha_1}\cdots z_d^{\alpha_d}.
\]
Indeed, the family $(e_{d,\alpha}^{(\lambda)})_{\alpha\in\bN_0^d}$ defined by
\[
e_{d,\alpha}^{(\lambda)}(z)=\sqrt{
\frac{\Gamma(d+|\alpha|+\lambda+1)}{\alpha!\Gamma(d+\lambda+1)}
}z^\alpha,\qquad z\in \bB^d,\ \alpha\in\bN_0^d,
\]
is an orthonormal basis of $\cA^2_\lambda(\bB^d)$. In particular, $e_{d,\alpha}^{(0)}$ denotes the corresponding normalized monomial in the unweighted space $\cA^2(\bB^d)$.

We will also need the analogous orthogonality relation on the sphere:
\begin{equation}
    \label{eq:int_S-zalpha-zbeta}
\int_{S^{2d-1}}\xi^\alpha\overline{\xi^\beta}d\sigma_d(\xi)=
\delta_{\alpha,\beta}\frac{2\pi^d\alpha!}{(d-1+|\alpha|)!}.
\end{equation}

For $\varphi\in L^\infty(\bB^d)$, the Toeplitz operator $T_{d,\varphi}$ with symbol $\varphi$ is defined by
\[
T_{d,\varphi}(f)=P_{\cA^2_\lambda(\bB^d)}(\varphi\cdot f),\qquad f\in \cA^2_\lambda(\bB^d).
\]
Whenever the dimension $d$ is clear from the context, we simply write $T_\varphi$ instead of $T_{d,\varphi}$.

Let $U(d)$ denote the group of $d\times d$ unitary matrices. This group acts naturally on functions $f\colon \bB^d\to\bC$ by
\[
A\cdot f(z)=f(A^{-1}z),\qquad z\in \bB^d,\ A\in U(d).
\]
Restricted to the Bergman space, this action yields a unitary representation
\[
\pi_{U(d)}\colon U(d) \to U(\cA^2_\lambda(\bB^d)),\quad
\pi_{U(d)}(A)(f)=A\cdot f,\ f\in \cA^2_\lambda(\bB^d),
\]
where $U(\cA^2_\lambda(\bB^d))$ denotes the group of unitary operators acting on $\cA^2_\lambda(\bB^d)$.
A direct computation shows that, for $\varphi\in L^\infty(\mathbb B^d)$,
\[
\pi_{U(d)}(A)T_\varphi \pi_{U(d)}(A)^*=T_{A\cdot \varphi}.
\]
Given a subgroup $G\subset U(d)$, let $\pi_G$ denote the restriction of $\pi_{U(d)}$ to $G$; thus,
$\pi_G(A)f=A\cdot f$ for $A\in G$ and $f\in \cA^2_\lambda(\bB^d)$.
A Toeplitz operator $T_\varphi$ commutes with $\pi_G$ if and only if its symbol $\varphi$ is $G$-invariant.

\subsection{Grouping of tuples and Bergman spaces on polyballs}

Let $n,m\in\bN$ with $m\leq n$, and let $k_1,\ldots,k_m\in \bN$ satisfy
$n=k_1+\cdots+k_m$. Throughout the paper, we partition each $n$-tuple
$u=(u_1,\ldots,u_n)$ into an $m$-tuple
$u=(u_{(1)},\ldots,u_{(m)})$, whose $j$-th component is a $k_j$-tuple. Thus,
\begin{align*}
u_{(1)}&=(u_1,u_2,\ldots,u_{k_1}),\\
u_{(2)}&=(u_{k_1+1},\ldots,u_{k_1+k_2}),\\
&\ \vdots\\
u_{(m)}&=(u_{k_1+\cdots+k_{m-1}+1},\ldots,u_{n}).
\end{align*}
For example, multi-indices $\alpha\in\bN_0^n$ will be written as
\[
\alpha=(\alpha_{(1)},\ldots,\alpha_{(m)}),
\]
where $\alpha_{(j)}\in \bN_0^{k_j}$, $j=1,\ldots,m$.

\medskip

For each $j\in\{1,\ldots,m\}$, consider the open ball $\bB^{k_j}$ in $\bC^{k_j}$.
We set
\begin{equation*}
    \bdB^{\bdk}=\bB^{k_1}\times\cdots\times\bB^{k_m},\quad
    \bdS^{\bdk}=
    S^{2k_1-1}\times\cdots\times S^{2k_m-1}.
\end{equation*}
We also consider the corresponding product measures:
\begin{align*}
    \bdV_{\bdk}&=V_{k_1}\times\cdots\times V_{k_m},\\
    \bdv_{\bdk}&=v_{k_1}\times\cdots\times v_{k_m},\\
    \bdsig_{\bdk}&=\sigma_{k_1}\times\cdots\times \sigma_{k_m}.
\end{align*}
Note that the measure $\bdv_{\bdk}$ is constructed from the unweighted measures. The construction readily extends to a product measure with different weight parameters on each factor, but we use the unweighted version for simplicity.

We denote by $\tau(\bB^m)$ the base of $\bB^m$ regarded as a Reinhardt domain. That is,
\[
\tau(\bB^m)=\big\{\bdr=(r_1,\ldots,r_m)\in [0,1]^m\colon r_1^2+\cdots+r_m^2<1\big\}.
\]
We will consider the usual Lebesgue measure $d\bdr=dr_1dr_2\cdots dr_m$ on $\tau(\bB^m)$.

Given $\bdr=(r_1,\ldots,r_m)\in [0,\infty)^m$ and $\bdxi=(\xi_1,\ldots,\xi_m)\in\bdS^{\bdk}$, we define
\[
\bdr\bdxi=(r_1\xi_1,\ldots,r_m\xi_m)\in \bC^n.
\]
For every $z=(z_{(1)},\ldots,z_{(m)})\in\bC^n$ with $|z_{(j)}|\neq 0$ for $j=1,\ldots,m$, we define $\bdr(z)\in [0,\infty)^m$ and $\bdxi(z)\in \bdS^{\bdk}$ by
\begin{equation}
    \label{eq:def-r(z)-xi(z)}
\bdr(z)\coloneqq(|z_{(1)}|,\ldots,|z_{(m)}|),\qquad \bdxi(z)\coloneqq\left(\frac{z_{(1)}}{|z_{(1)}|},\ldots,\frac{z_{(m)}}{|z_{(m)}|}\right),
\end{equation}
which yields $z=\bdr(z)\bdxi(z)$.
We will repeatedly use these coordinate identifications to represent functions. As is customary, in this paper, all identities involving $L^\infty$-symbols are understood almost everywhere. Whenever angular sections are used, we fix representatives on a common conull subset.

These coordinates yield the following natural identifications up to null sets:
\[
\bB^n\longleftrightarrow \tau(\bB^m)\times\bdS^{\bdk},\qquad
\bdB^{\bdk}\longleftrightarrow
[0,1)^m\times\bdS^{\bdk}.
\]
As above, we also use the shorthand notation
\[
\bdr^{p}=r_1^{p_1}\cdots r_m^{p_m},\qquad
\bdr\in\tau(\bB^m),\ p\in \bN_0^m.
\]
We also set
\[
|\bdr|_2^2=r_1^2+\cdots+r_m^2.
\]
Applying \eqref{eq:measures-polar-coordinates} to each component gives, for every $f\in L^1(\bB^n,dv_{n,\lambda})$,
\begin{equation}
    \label{eq:int-formula-m-polar-coord}
\int_{\bB^n}f(z)
dv_{n,\lambda}(z)
=
c_{n,\lambda}
\int_{\tau(\bB^m)\times\bdS^{\bdk}}
f(\bdr\bdxi)
(1-|\bdr|_2^2)^\lambda \bdr^{2\bdk-\bdone}
d\bdr\bdsig_{\bdk}(\bdxi),
\end{equation}
where $\bdone=(1,\ldots,1)\in\bN_0^m$.

Let $\cA^2(\bdB^{\bdk})$ be the Bergman space on the polyball $\bdB^{\bdk}$ defined as the space of all holomorphic functions that are square integrable with respect to the measure $\bdv_{\bdk}$.
Standard arguments show that $\cA^2(\bdB^{\bdk})$ is a reproducing kernel Hilbert space, embedded as a closed subspace in $L^2(\bdB^{\bdk},\bdv_{\bdk})$.

As before, let $P_{\cA^2(\bdB^{\bdk})}$ denote the orthogonal projection from $L^2(\bdB^{\bdk},\bdv_{\bdk})$ onto $\cA^2(\bdB^{\bdk})$. For $\varphi\in L^\infty(\bdB^{\bdk})$, we denote by $\bdT_{\varphi}$ the Toeplitz operator
\[
\bdT_{\varphi}f=P_{\cA^2(\bdB^{\bdk})}(\varphi f),
\qquad f\in \cA^2(\bdB^{\bdk}).
\]

The space $\cA^2(\bdB^{\bdk})$ has the orthonormal basis $(\bde_{\bdk,\alpha}^{(0)})_{\alpha\in \bN_0^n}$ defined by
\[
\bde_{\bdk,\alpha}^{(0)}\coloneqq e_{k_1,\alpha_{(1)}}^{(0)}\otimes\cdots\otimes e_{k_m,\alpha_{(m)}}^{(0)}.
\]
That is,
\[
\bde_{\bdk,\alpha}^{(0)}(z)
=
e_{k_1,\alpha_{(1)}}^{(0)}(z_{(1)})\cdots
e_{k_m,\alpha_{(m)}}^{(0)}(z_{(m)}),
\]
where $z=(z_{(1)},\ldots,z_{(m)})\in \bdB^{\bdk}$ and $\alpha=(\alpha_{(1)},\ldots,\alpha_{(m)})\in \bN_0^{k_1}\times\cdots\times\bN_0^{k_m}$.
Consequently, we obtain the natural tensor-product identification
\[
\cA^2(\bdB^{\bdk})\cong
\cA^2(\bB^{k_1})\otimes\cdots\otimes\cA^2(\bB^{k_m}).
\]
Throughout the polyball constructions and local models below,
$T_{k_j,\varphi_j}$ denotes a Toeplitz operator on the unweighted space
$\cA^2(\bB^{k_j})$. The operators on the original ball
$\cA^2_\lambda(\bB^n)$ retain the fixed weight $\lambda$.
If $\varphi\in L^\infty(\bdB^{\bdk})$ is of the form $\varphi = \varphi_1\otimes\cdots\otimes\varphi_m$, where $\varphi_j\in L^\infty(\bB^{k_j})$, then
\[
\bdT_{\varphi}=T_{k_1,\varphi_1}\otimes\cdots\otimes T_{k_m,\varphi_m}.
\]

\subsection{Spaces of homogeneous polynomials}

For $\ka\in \mathbb N_0^m$, define the subspace $H_{n,\ka}\subset\cA^2_\lambda(\bB^n)$ by
\begin{equation}
    \label{eq:def-Hnka}
    H_{n,\kappa} = \operatorname{span}\{e_{n,\alpha}^{(\lambda)}\colon \alpha\in\bN_0^n,\  |\alpha_{(j)}|=\kappa_j,\ j=1,\ldots,m\},
\end{equation}
and the subspace $\bdH_{\bdk,\ka}\subset\cA^2(\bdB^{\bdk})$ by
\begin{equation}
    \label{eq:def-bold-Hkka}
    \bdH_{\bdk,\ka} =
    \operatorname{span}\{\bde_{\bdk,\alpha}^{(0)}\colon
    \alpha\in \bN_0^n,\ |\alpha_{(j)}|=\kappa_j,\ j=1,\ldots,m\}.
\end{equation}
There is a natural identification
\[
\bdH_{\bdk,\ka}\cong H_{k_1,\ka_1}\otimes\cdots\otimes H_{k_m,\ka_m},
\]
where for each $j\in\{1,\ldots,m\}$ and $d\in\bN_0$, the space $H_{k_j,d}$ is the analogous subspace of $\cA^2(\bB^{k_j})$ defined by
\[
H_{k_j,d}=
\operatorname{span}
\{e_{k_j,\alpha_{(j)}}^{(0)}\colon
\alpha_{(j)}\in\bN_0^{k_j},\ |\alpha_{(j)}|=d\}.
\]

Because the corresponding normalized monomials form orthonormal bases of the two Bergman spaces, we obtain
\[
\cA^2_\lambda(\bB^n)=
\bigoplus_{\ka\in\bN_0^m}
H_{n,\ka},\qquad
\cA^2(\bdB^{\bdk})
=
\bigoplus_{\ka\in\bN_0^m}
\bdH_{\bdk,\ka}.
\]
Consequently, there is a unique unitary operator
\begin{equation}
    \label{eq:def-op-U}
    U\colon \cA^2_\lambda(\bB^n)\longrightarrow \cA^2(\bdB^{\bdk})
\end{equation}
such that
\[
U(e_{n,\alpha}^{(\lambda)})=
\bde_{\bdk,\alpha}^{(0)},\qquad \alpha\in \bN_0^n.
\]
For each $d\in \bN_0$ and $j\in\{1,\ldots,m\}$, the space $H_{k_j,d}$ is naturally isomorphic to the space $\cP_d(\bC^{k_j})$ of homogeneous polynomials of degree $d$.
Therefore, for every
$\ka\in\bN_0^m$,
both spaces $H_{n,\ka}$ and $\bdH_{\bdk,\ka}$ are naturally isomorphic to the space $\bigotimes_{j=1}^m \cP_{\ka_j}(\bC^{k_j})$ of ``quasi-homogeneous'' polynomials.

\subsection{Bounded operators and operator algebras}

Throughout this paper, we work with operators and operator algebras defined on several spaces. We conclude this section by fixing notation and recalling a few standard facts. For further details, see, for example, \cite{Kaniuth,RodriguezRodriguezPHD2024}.

For a Hilbert space $H$, let $\cL(H)$ denote the $C^*$-algebra of bounded operators on $H$ and $\cK(H)$ its closed $*$-ideal of compact operators.
If $\cT\subset\cL(H)$ is a unital Banach algebra and $\cJ$ is a closed ideal of $\cT$, then $\cT/\cJ$ is again a unital Banach algebra. If $\cT$ is a $C^*$-algebra and $\cJ$ is a closed $*$-ideal, then the quotient is a $C^*$-algebra.
In particular, we define the Calkin algebra of $H$ as the quotient algebra $\cL(H)/\cK(H)$. Its elements are denoted by $[T]_\cK\coloneqq T+\cK(H)$, for $T\in\cL(H)$.

Given an operator $T\in \cL(H)$, we denote by $\operatorname{sp}(T)$ its usual spectrum with respect to the $C^*$-algebra $\cL(H)$. The spectrum of an element $[T]_\cK$ in the Calkin algebra is called the essential spectrum of $T$ and is denoted by $\operatorname{ess-sp}(T)$.

If $S$ is a subset of a unital Banach algebra $\cA$, let $\cB(S)$ denote the unital Banach algebra generated by $S$ in $\cA$. When $\cA$ is a $C^*$-algebra, $C^*(S)$ denotes the $C^*$-algebra generated by $S$. For a singleton $S=\{T\}$, we simply write $\cB(T)$ and $C^*(T)$.

We finally recall polynomial convexity, which is closely related to the theory of commutative Banach algebras.
If $K\subset\bC^u$ is compact, we define its polynomially convex hull $\pconv{K}$ as the set of all $z\in\bC^u$ such that, for every holomorphic polynomial $p$ in $u$ variables,
\[
|p(z)|\leq \sup_{w\in K}|p(w)|.
\]
A compact set $K$ is called polynomially convex if $K=\pconv{K}$.
If $T$ is an element of some unital Banach algebra $\cA$, then the maximal ideal space $M(\cB(T))$ of $\cB(T)$ is homeomorphic to the polynomially convex hull of its spectrum with respect to $\cA$:
\[
M(\cB(T))\cong\pconv{\operatorname{sp}_{\cA}(T)}.
\]
If $A_j\subset\bC^{d_j}$ is compact for $j=1,\ldots,L$, then
\[
\pconv{A_1\times\cdots\times A_L}
=
\pconv{A_1}\times\cdots\times \pconv{A_L}.
\]

\section{\texorpdfstring{$\mathbb T^m$-invariant operators}{Tm invariant operators}}
\label{sec:Tm-inv-Toep-op}

\subsection{\texorpdfstring{$U(\bdk)$- and $\bT^m$-invariant operators}{U(k)- and Tm-invariant operators}}
\label{sec:U(k)-and-Tm-invariant-operators}

In this section, we recall several facts about $k$-quasi-radial symbols and their associated Toeplitz operators.

Let $n\in \bN$ and consider the Bergman space $\cA^2_\lambda(\bB^n)$. As before, we let $m$ and $k_1,\ldots,k_m$ be positive integers with $n=k_1+\cdots+k_m$. We define $\bdk=(k_1,\ldots,k_m)$.
As in the preceding section, we partition the coordinates of $\mathbb C^n$ into $m$ blocks:
\[
z=(z_{(1)},z_{(2)},\ldots,z_{(m)})\in \mathbb C^{k_1}\times \cdots \times \mathbb C^{k_m}=\bC^n.
\]
With respect to this decomposition, consider the direct product of unitary groups
\[
U(\bdk)\coloneqq U(k_1)\times\cdots\times U(k_m),
\]
where, as before, $U(d)$ denotes the unitary group of $d\times d$ matrices acting on $\mathbb C^d$. Via block-diagonal matrices, the group $U(\bdk)$ is naturally identified with a subgroup of $U(n)$ acting on $\mathbb C^n$.

As is well known, for each $j\in \{1,\ldots,m\}$ the space of homogeneous polynomials of any fixed degree in $\bC^{k_j}$ is an irreducible module for the unitary group $U(k_j)$. Therefore, the space $\bigotimes_{j=1}^m \cP_{\ka_j}(\bC^{k_j})$ is an irreducible module for $\pi_{U(\bdk)}$.

Next, identify $\bT^m$ with the subgroup of $U(\bdk)$ consisting of the block-diagonal matrices
\begin{equation}
\label{eq:explicit_action_Tm}
\begin{pmatrix}
t_1I_{k_1} &        & \\
           & \ddots & \\
           &        & t_mI_{k_m}
\end{pmatrix},
\end{equation}
where $I_{k_j}$ is the $k_j\times k_j$ identity matrix for $j=1,\ldots,m$.

The isotypic decompositions of these representations yield the following proposition (see \cite{Quiroga2021}).
\begin{prop}
\label{prop:T-Uk-Tm-as-direct-sum}
    Let $T\in \cL(\cA^2_\lambda(\bB^n))$.
    If $T$ commutes with $\pi_{U(\bdk)}$ or with $\pi_{\bT^m}$, then $T$ leaves the spaces $H_{n,\ka}$ invariant for every $\ka\in\bN_0^m$ and
    \[
    T=\bigoplus_{\ka\in\bN_0^m}T|_{H_{n,\ka}}.
    \]
    Moreover, if $T$ commutes with $\pi_{U(\bdk)}$, then $T|_{H_{n,\ka}}$ is a multiple of the identity $I_{H_{n,\ka}}$ for every $\ka\in\bN_0^m$.
\end{prop}

By the discussion in the preceding section, a Toeplitz operator $T_\varphi$ commutes with $\pi_{U(\bdk)}$ if and only if $\varphi$ is $U(\bdk)$-invariant, and it commutes with $\pi_{\bT^m}$ if and only if $\varphi$ is $\bT^m$-invariant. These classes of symbols play an important role in this work. Following the notation in \cite{Quiroga2021}, given a subgroup $G\subset U(n)$ we let $L^\infty(\bB^n)^G$ denote the set of bounded measurable functions invariant under $G$. That is,
\[
\begin{aligned}
L^\infty(\bB^n)^G
={}&\left\{\varphi\in L^\infty(\bB^n)\colon\right.\\
&\left.\qquad
\varphi(A^{-1}z)=\varphi(z)\text{ for a.e. }z\in\bB^n
\text{ for every }A\in G
\right\}.
\end{aligned}
\]
The functions in $\LinftyUk$ are called \emph{$k$-quasi-radial}, whereas those in $\LinftyTm$ are called \emph{$\bT^m$-invariant}.

We define the analogous spaces on the polyball. In this setting, we restrict attention to subgroups $G\subset U(\bdk)$,
which preserve $\bdB^{\bdk}$. For such $G$, we set
\[
\begin{aligned}
L^\infty(\bdB^{\bdk})^G
={}&\left\{\varphi\in L^\infty(\bdB^{\bdk})\colon\right.\\
&\left.\qquad
\varphi(A^{-1}z)=\varphi(z)\text{ for a.e. }z\in\bdB^{\bdk}
\text{ for every }A\in G
\right\}.
\end{aligned}
\]
All groups considered below are subgroups of $U(\bdk)$.

The space $\LinftyUk$, the Toeplitz operators with symbols in this space, and analogous constructions on other function spaces have been studied extensively (see, for example, \cite{Vasilevski2010,BauerVasilevski2013}).
A direct verification shows that $U(\bdk)$-invariance is equivalent to the requirement that $a$ depend only on the grouped radii $(|z_{(1)}|,\ldots,|z_{(m)}|)$ of $z$. Indeed (see \cite{Quiroga_Sanchez_2021}), a function $a$ is $k$-quasi-radial if there exists a function $h_a$ defined on $\tau(\bB^m)$ such that
\[
a(z)=h_a(|z_{(1)}|,\ldots,|z_{(m)}|),\qquad z\in\mathbb B^n.
\]
For simplicity, we identify the functions $a$ and $h_a$ whenever no confusion arises.

The preceding observations, together with a direct computation, yield the following standard result (see \cite{Vasilevski2010,Quiroga2021}):
\begin{prop}
\label{prop:Ta-quasi-radial}
Let $a\in L^\infty(\mathbb B^n)$. Then $a$ is $k$-quasi-radial if and only if, for every $\ka\in\mathbb N_0^m$, the subspace $H_{n,\ka}$ is invariant under $T_{n,a}$ and the restriction $T_{n,a}|_{H_{n,\ka}}$ is a scalar multiple of the identity:
\[
T_{n,a}|_{H_{n,\ka}}=\gamma_{a,\lambda}(\ka) I_{H_{n,\ka}}.
\]
In this case, the eigenvalue sequence $\gamma_{a,\lambda}$ is given by
\begin{equation}
\label{eq:formula-gamma_a}
\gamma_{a,\lambda}(\kappa)=
\frac{2^m\Gamma(n+\lambda+|\kappa|+1)}{\Gamma(\lambda+1)
(\bdk+\ka-\bdone)!}
\int_{\tau(\mathbb B^m)}
a(\bdr)(1-|\bdr|_2^2)^\lambda
\bdr^{2\bdk+2\ka-\bdone}d\bdr.
\end{equation}
\end{prop}

\subsection{\texorpdfstring{The commutative $C^*$-algebra $\Tkqr$}{The commutative C-star algebra Tk-qr}}

As a consequence of Proposition~\ref{prop:Ta-quasi-radial}, the $C^*$-algebra
\[
\Tkqr=
C^*(T_a\colon a\text{ is }k\text{-quasi-radial}),
\]
generated by all Toeplitz operators with $k$-quasi-radial symbols, is a unital commutative $C^*$-algebra. Let $M(\Tkqr)$ denote its maximal ideal space.
The elements of $\Tkqr$ are diagonal operators $D_\gamma$ of the form
\[
D_\gamma=\bigoplus_{\ka\in\bN_0^m}\gamma(\ka)I_{H_{n,\ka}},
\]
where $\gamma$ is a bounded sequence defined on $\bN_0^m$.

We will need a fairly detailed description of $M(\Tkqr)$ in the following sections. For a complete account, see \cite{BauerVasilevski2013}. To streamline the notation, we use an approach that is slightly different from, but equivalent to, the one adopted there.

Under the identification $D_\gamma\cong \gamma$, we regard the algebra $\mathcal T_{k\text{-qr}}$ as a $C^*$-algebra of bounded functions on $\mathbb N_0^m$.
It was shown in \cite{BauerVasilevski2013} that the algebra $\mathcal T_{k\text{-qr}}$ contains all functions on $\mathbb N_0^m$ with limits at infinity. In particular, it contains all orthogonal projections
\[
P_\kappa \colon \mathcal A^2_\lambda(\mathbb B^n)\longrightarrow H_{n,\kappa},
\]
for $\kappa\in \mathbb N_0^m$.
Since these projections separate the points of $\mathbb N_0^m$, we conclude that $M(\mathcal T_{k\text{-qr}})$ is homeomorphic to a compactification of $\mathbb N_0^m$.
We henceforth regard $M(\Tkqr)$ as this compactification and denote by $\psi_\mu$ the character corresponding to $\mu\in M(\Tkqr)$. Thus,
\[
\psi_\mu(D_\gamma)=\gamma(\mu),\qquad D_\gamma\in \Tkqr.
\]

It was also proved in \cite{BauerVasilevski2013} that, for every $j\in\{1,\ldots,m\}$ and $d\in\bN_0$, the projections
\begin{equation}
\label{eq:def-Qjd}
Q_{j,d}=\bigoplus_{\ka\in\bN_0^m,\ \ka_j=d}P_{\ka}
\end{equation}
belong to $\Tkqr$ as well.
Under the unitary isomorphism \eqref{eq:def-op-U}, we have
\[
UQ_{j,d}U^*=
I\otimes\cdots\otimes P_{j,d}\otimes\cdots\otimes I,
\]
and
\[
UP_\ka U^*=P_{1,\ka_1}\otimes\cdots\otimes
P_{m,\ka_m},
\]
where $P_{j,d}$ is the orthogonal projection from $\cA^2(\bB^{k_j})$ onto $H_{k_j,d}$.

Now let $\Omega=(\mathbb N_0\cup\{\infty\})^m$ be the compactification of $\mathbb N_0^m$ by adding a ``point at infinity'' for each coordinate.
The projections $Q_{j,d}$ defined above show that the space $C(\Omega)$ is continuously embedded into $\mathcal T_{k\text{-qr}}$. In particular, restriction induces the continuous surjection
\[
M(\mathcal T_{k\text{-qr}})
\longrightarrow \Omega,
\]
given by the restriction $\psi\mapsto \psi|_{C(\Omega)}$.
This yields the stratification
\begin{equation}
\label{eq:M=bicup-Momega}
M(\mathcal T_{k\text{-qr}})=
\bigcup_{\omega\in \Omega}M(\omega),
\end{equation}
where $\mu\in M(\omega)$ if and only if, for every $j\in\{1,\ldots,m\}$ and $d\in\bN_0$,
\[
\begin{cases}
\psi_\mu(Q_{j,d})=0,&\quad \text{if }\omega_j=\infty\\ \psi_\mu(Q_{j,d})=\delta_{d,\omega_j},&\quad \text{if }\omega_j\in\bN_0.
\end{cases}
\]
The points of $\Omega$ encode the coordinatewise directions in which nets in $\mathbb N_0^m$ approach points of $M(\mathcal T_{k\text{-qr}})$. For example, for every $\kappa\in \mathbb N_0^m\subset \Omega$ we have
\[
M(\kappa)=\{\kappa\},
\]
while the set $M((\infty,\infty,\ldots,\infty))$ is the collection of points in $M(\mathcal T_{k\text{-qr}})$ that are limits of nets whose coordinates all tend to $\infty$.
Accordingly, we refer to the $j$th component of $\mu\in M(\Tkqr)$ as the $j$th component of $\omega$ such that $\mu\in M(\omega)$.

The algebra $\Tkqr$ is much larger than $C(\Omega)$ and contains functions that oscillate slowly in a specific sense described in \cite{BauerVasilevski2013}. In fact, the topology of $M(\Tkqr)$ is most naturally described in terms of nets.

\subsection{\texorpdfstring{$\bT^m$-invariant Toeplitz operators}{Tm-invariant Toeplitz operators}}
\label{sec:Tm-invariant-Toep-op}

In this section, we show that Toeplitz operators can naturally be viewed as families of Toeplitz operators on the Bergman space over the polyball indexed by the space $M(\Tkqr)$.
Under suitable conditions on the symbols, these families are continuous.

Let $a\in \LinftyTm$. As in the quasi-radial case, this implies the existence of a function $f_a\in L^\infty(\tau(\bB^m)\times \bdS^{\bdk})$ such that
\begin{equation}
    \label{eq:def-associated-fa-Tm-invariant}
    a(z)=f_a\big(\bdr(z),\bdxi(z)\big)
\end{equation}
and
\[
f_a(\bdr,t_1\xi_1,\ldots,t_m\xi_m)=
f_a(\bdr,\xi_1,\ldots,\xi_m),
\]
for almost every $\bdr\in\tau(\bB^m)$, every $t=(t_1,\ldots,t_m)\in\bT^m$ and almost every $(\xi_1,\ldots,\xi_m)\in \bdS^{\bdk}$.
Note that, by \eqref{eq:explicit_action_Tm}, we mean
\[
t_j\xi_j=(t_j\xi_{j,1},t_j\xi_{j,2},\ldots,t_j\xi_{j,k_j}),
\]
where $\xi_j=(\xi_{j,1},\ldots,\xi_{j,k_j})\in S^{2k_j-1}$, for $j=1,\ldots,m$.
In contrast to the quasi-radial symbols, in this case, to avoid ambiguity, we will distinguish between $a$ and $f_a$.

\begin{thm}
\label{thm:Tna=oplus-T-polyball}
Let $a\in L^\infty(\bB^n)^{\bT^m}$. Then, under the unitary operator $U$,
\[
UT_{n,a}U^*=\bigoplus_{\ka\in \bN_0^m} \bdT_{\widetilde{a}_\ka}|_{\bdH_{\bdk,\ka}},
\]
where $\widetilde{a}_\ka\in \LinftyTmpol$ is defined by
\[
\widetilde{a}_\ka(z)=
\gamma_{f_a(\cdot,\bdxi(z)),\lambda}(\ka),
\]
for $z\in\bdB^{\bdk}$. Here, $f_a$ is the representative associated with $a$ in \eqref{eq:def-associated-fa-Tm-invariant}; explicitly,
\begin{equation}
\label{eq:def-a-tilde-kappa}
\widetilde{a}_\ka(z)=
\frac{2^m\Gamma(n+|\kappa|+\lambda+1)}{\Gamma(\lambda+1)(\bdk+\ka-\bdone)!}
\int_{\tau({\mathbb B}^m)}
f_a\big(\bdr,\bdxi(z)\big)
\bdr^{2\bdk+2\ka-\bdone}(1-|\bdr|_2^2)^\lambda d\bdr.
\end{equation}
Moreover, for every $\ka\in\bN_0^m$, the following properties hold:
\begin{enumerate}
    \item $\|\widetilde{a}_\ka\|_\infty\leq \|a\|_{\infty}$.
    \item The symbol $\widetilde{a}_\ka$ depends only on its $\bdS^{\bdk}$-component: there exists a function $g_{a,\ka}\in L^\infty(\bdS^{\bdk})$, unique up to equality almost everywhere, such that
    \[
    \widetilde{a}_\ka(\bdr\bdxi)=g_{a,\ka}(\bdxi),
    \]
    for almost every $(\bdr,\bdxi)\in[0,1)^m\times\bdS^{\bdk}$.
\end{enumerate}
\end{thm}
\begin{proof}
    Since $a\in \LinftyTm$, the Toeplitz operator $T_{n,a}$ commutes with the representation $\pi_{\bT^m}$. Therefore, Proposition~\ref{prop:T-Uk-Tm-as-direct-sum} gives
    \[
    T_{n,a}=\bigoplus_{\ka\in\bN_0^m}T_{n,a}|_{H_{n,\ka}}.
    \]

    Fix $\ka\in\bN_0^m$ and $\alpha=(\alpha_{(1)},\ldots,\alpha_{(m)}),\beta=(\beta_{(1)},\ldots,\beta_{(m)})\in \bN_0^n$ with $|\alpha_{(j)}|=|\beta_{(j)}|=\ka_j$, $j=1,\ldots,m$.
    Applying \eqref{eq:int-formula-m-polar-coord} and Fubini's theorem gives
    \begin{align}
    \label{eq:Ta-tm-inv-prod-1-int}
        \langle
        T_{n,a}e_{n,\alpha}^{(\lambda)},e_{n,\beta}^{(\lambda)}
        \rangle&=
        \int_{\bB^n}a(z)e_{n,\alpha}^{(\lambda)}(z)\overline{e_{n,\beta}^{(\lambda)}(z)}dv_{n,\lambda}(z)\nonumber\\
        &=
        \frac{\Gamma(n+|\kappa|+\lambda+1)}
        {\pi^n\Gamma(\lambda+1)\sqrt{\alpha!\beta!}}
        \int_{{\mathbb B}^n}
        a(z)
        z^\alpha\overline{z^\beta}
        (1-|z|^2)^\lambda dV_n(z)\nonumber\\
        &=
        \frac{\Gamma(n+|\kappa|+\lambda+1)}
        {\pi^n\Gamma(\lambda+1)\sqrt{\alpha!\beta!}}
        \int_{\tau(\bB^m)}
        (1-|\bdr|_2^2)^\lambda \bdr^{2\bdk+2\ka-\bdone}d\bdr\nonumber\\
        &\times
        \int_{\bdS^{\bdk}}f_a(\bdr,\bdxi)\bdxi^\alpha\overline{\bdxi^\beta}d\bdsig_{\bdk}(\bdxi)\nonumber\\
        &=
        \frac{\Gamma(n+|\kappa|+\lambda+1)}
        {\pi^n\Gamma(\lambda+1)\sqrt{\alpha!\beta!}}
        \int_{\bdS^{\bdk}}F_a(\bdxi)\bdxi^\alpha\overline{\bdxi^\beta}d\bdsig_{\bdk}(\bdxi),
    \end{align}
where
\[
    F_a(\bdxi)=\int_{\tau(\bB^m)}
        f_a(\bdr,\bdxi)
        (1-|\bdr|_2^2)^\lambda \bdr^{2\bdk+2\ka-\bdone}d\bdr.
\]
The identity $\prod_{j=1}^m\left(
2(k_j+\ka_j)
\int_0^1t_j^{2k_j+2\ka_j-1}dt_j\right)=1$ yields
\begin{align*}
\int_{\bdS^{\bdk}}F_a(\bdxi)\bdxi^\alpha\overline{\bdxi^\beta}d\bdsig_{\bdk}(\bdxi)
&=
2^m
(\bdk+\ka)^{\bdone}
\int_{[0,1]^m}\bdt^{2\bdk-\bdone}d\bdt\notag\\
&\quad\times
\int_{\bdS^{\bdk}}
F_a(\bdxi)(\bdt\bdxi)^\alpha\overline{(\bdt\bdxi)^\beta}
d\bdsig_{\bdk}(\bdxi),
\end{align*}
where $\bdt=(t_1,\dots,t_m)$ and $\bdt\bdxi=(t_1\xi_1,\ldots,t_m\xi_m)\in\bdB^{\bdk}$. Therefore, applying the spherical-coordinate formula \eqref{eq:measures-polar-coordinates} to each factor gives
\[
\int_{\bdS^{\bdk}}
F_a(\bdxi)\bdxi^\alpha\overline{\bdxi^\beta}
d\bdsig_{\bdk}(\bdxi)
=
2^m(\bdk+\ka)^{\bdone}
\int_{\bdB^{\bdk}}F_a\left(\frac{z_{(1)}}{|z_{(1)}|},\ldots,\frac{z_{(m)}}{|z_{(m)}|}\right)z^\alpha\overline{z^\beta}d\bdV_{\bdk}(z).
\]
Substituting this identity into \eqref{eq:Ta-tm-inv-prod-1-int} gives
\begin{align*}
    \langle
    T_{n,a}e_{n,\alpha}^{(\lambda)},e_{n,\beta}^{(\lambda)}
    \rangle&=
    \frac{2^m\Gamma(n+|\kappa|+\lambda+1)}
        {\pi^n\Gamma(\lambda+1)\sqrt{\alpha!\beta!}}
        (\bdk+\ka)^{\bdone}\\
&
\times\int_{\bdB^{\bdk}}F_a\big(\bdxi(z)\big)z^\alpha\overline{z^\beta}d\bdV_{\bdk}(z)
\\
&=
\frac{2^m\Gamma(n+|\kappa|+\lambda+1)}
        {\pi^n\Gamma(\lambda+1)\sqrt{\alpha!\beta!}}
        (\bdk+\ka)^{\bdone}
        \prod_{j=1}^m\frac{\sqrt{\alpha_{(j)}!\beta_{(j)}!}k_j!}{(k_j+\ka_j)!}\\
&\times\int_{\bdB^{\bdk}}F_a\big(\bdxi(z)\big)\bde_{\bdk,\alpha}^{(0)}(z)\overline{\bde_{\bdk,\beta}^{(0)}(z)}d\bdV_{\bdk}(z)\\
&=
\frac{2^m\Gamma(n+|\kappa|+\lambda+1)}
        {\pi^n\Gamma(\lambda+1)}
        \prod_{j=1}^m\frac{k_j!}{(k_j+\ka_j-1)!}\\
&\times\int_{\bdB^{\bdk}}F_a\big(\bdxi(z)\big)\bde_{\bdk,\alpha}^{(0)}(z)\overline{\bde_{\bdk,\beta}^{(0)}(z)}d\bdV_{\bdk}(z)\\
&=
\frac{2^m\Gamma(n+|\kappa|+\lambda+1)}{\Gamma(\lambda+1)(\bdk+\ka-\bdone)!}
\int_{\bdB^{\bdk}}F_a\big(\bdxi(z)\big)\bde_{\bdk,\alpha}^{(0)}(z)\overline{\bde_{\bdk,\beta}^{(0)}(z)}d\bdv_{\bdk}(z)\\
&=
\int_{\bdB^{\bdk}}\gamma_{f_a(\cdot,\bdxi(z)),\lambda}(\ka)\bde_{\bdk,\alpha}^{(0)}(z)\overline{\bde_{\bdk,\beta}^{(0)}(z)}d\bdv_{\bdk}(z).
\end{align*}
The $\bT^m$-invariance of $\widetilde{a}_\ka$ follows from the invariance of $a$; the remaining assertions follow from \eqref{eq:def-a-tilde-kappa}.
\end{proof}

We finish this section by providing some examples. Different versions and extensions of these are easily constructed.

\begin{ex}
\label{ex:example1-symbol}
Consider the case $n=4$, $\lambda=0$, $m=2$ and $\bdk=(2,2)$. Hence, for any $z\in\bB^4$, we have the tuples
\[
z_{(1)}=(z_1,z_2),\ z_{(2)}=(z_3,z_4)\in \bC^2,
\]
\[
\bdr(z)=(|z_{(1)}|,|z_{(2)}|)\in \tau(\bB^2),
\]
\[
\bdxi(z)=
\left(
\frac{z_{(1)}}{|z_{(1)}|},
\frac{z_{(2)}}{|z_{(2)}|}
\right)
\in S^{3}\times S^{3}.
\]
For $j=1,2$, let $H_j\in L^\infty(S^3\times S^3)$ be any real-valued function invariant under the action of $\bT^2$.
For example, we may take a function of the form 
\[
H_j(\xi_1,\xi_2)=h_j\big(\xi_{1,1}\overline{\xi_{1,2}},\xi_{2,1}\overline{\xi_{2,2}}\big),\qquad \xi_1,\xi_2\in S^3,
\]
for some real-valued bounded Borel function $h_j$.

We define the symbol $\phi\in L^\infty(\bB^4)$ by
$\phi(z)=f_\phi(\bdr(z),\bdxi(z))$, $z\in \bB^4$,
where
$f_\phi\in L^\infty(\tau(\bB^2)\times\bdS^{\bdk})$ is the function
\[
f_\phi(\bdr,\bdxi)=r_1^{iH_1(\bdxi)}
r_2^{i H_2(\bdxi)}.
\]
Then, for $\ka\in\bN_0^m$, one has from \eqref{eq:def-a-tilde-kappa} and Fubini's theorem
\begin{align}
\label{eq:example1-formula}
&\widetilde{\phi}_\ka(z)=
\frac{2^m(|\ka|+4)!}{(\ka_1+1)!(\ka_2+1)!}
\int_{\tau(\bB^2)}
r_1^{iH_1(\bdxi)}
r_2^{iH_2(\bdxi)}
r_1^{2\ka_1+3}r_2^{2\ka_2+3}dr_1dr_2\nonumber
\\
&=
\frac{|\ka|+4}
{
\frac{i}{2}H_1(\bdxi)+
\frac{i}{2}H_2(\bdxi)+|\ka|+4
}
\cdot
\frac{B
\bigg(
\frac{i}{2}H_1(\bdxi)+\ka_1+2,
\frac{i}{2}H_2(\bdxi)+\ka_2+2
\bigg)}{B(\ka_1+2,\ka_2+2)},
\end{align}
where $B(z,w)$ denotes the Beta function.
\end{ex}

\begin{ex}
\label{ex:example2-symbol}
    In the same setting as above, let now $G_1,G_2\in L^\infty(S^3\times S^3)$ be any (complex-valued) functions invariant under $\bT^2$, and let $\alpha,\beta\geq 0$.
    
    Consider the symbol $\psi\in L^\infty(\bB^4)$ such that
    \[
    f_\psi(\bdr,\bdxi)=
    r_1^{2\alpha} G_1(\bdxi)
    +r_2^{2\beta} G_2(\bdxi).
    \]
    Then
    \begin{align}
    \label{eq:example2-formula}
    \widetilde{\psi}_\ka(z)&=
    \frac{2^m(|\ka|+4)!}{(\ka_1+1)!(\ka_2+1)!}
    \int_{\tau(\bB^2)}
    \bigg(
    r_1^{2\alpha} G_1(\bdxi)
    +r_2^{2\beta} G_2(\bdxi)
    \bigg)
    r_1^{2\ka_1+3}r_2^{2\ka_2+3}dr_1dr_2 \nonumber\\
    &=
    \frac{B(|\ka|+5,\alpha)}{B(\ka_1+2,\alpha)}
    G_1(\bdxi)+
    \frac{B(|\ka|+5,\beta)}{B(\ka_2+2,\beta)}
    G_2(\bdxi),
    \end{align}
    where the cases $\alpha=0$ and $\beta=0$ are understood in the limiting sense:
    \[
    \lim_{\alpha\to 0}\frac{B(|\ka|+5,\alpha)}{B(\ka_1+2,\alpha)}=\lim_{\beta\to0}\frac{B(|\ka|+5,\beta)}{B(\ka_2+2,\beta)}=1.
    \]
    \end{ex}

\subsection{Continuous families of operators}
\label{sec:cont-families}

Theorem~\ref{thm:Tna=oplus-T-polyball} shows that, given $a\in \LinftyTm$, one gets a family of operators through the map
\[
\mathbb N_0^m\ni \kappa \longmapsto
\bdT_{\widetilde{a}_\ka}|_{\bdH_{\bdk,\ka}},
\]
where the operators $\bdT_{\widetilde{a}_\ka}$ act on the Bergman space on the polyball $\cA^2(\bdB^{\bdk})$.

Since $\widetilde{a}_\ka(z)$ depends only on the $\bdS^{\bdk}$-component of $z$, we may regard it as a function on $\bdS^{\bdk}$ by setting
\[
\widetilde{a}_\ka|_{\bdS^{\bdk}}(\bdxi)\coloneqq \widetilde{a}_\ka(\bdr\bdxi),
\]
for any $\bdr\in(0,1)^m$.

If the original symbol $a\in L^\infty(\bB^n)$ also depends only on its $\bdS^{\bdk}$-component, we may define $a|_{\bdS^{\bdk}}$ on $\bdS^{\bdk}$ in a similar way and then \eqref{eq:def-a-tilde-kappa} implies
\[
a|_{\bdS^{\bdk}}(\bdxi)=\widetilde{a}_\ka|_{\bdS^{\bdk}}(\bdxi),
\]
for every $\ka\in \bN_0^m$ and almost every $\bdxi\in\bdS^{\bdk}$.

Equation \eqref{eq:formula-gamma_a} shows that, for almost every fixed $z\in\bdB^{\bdk}$,
the function
\[
\bN_0^m\ni\kappa\longmapsto\widetilde{a}_{\ka}(z)
\]
is the eigenvalue sequence associated with a $k$-quasi-radial symbol. We denote this function by $\widetilde{a}_\bullet(z)$.

If $\mu\in M(\Tkqr)$ and $\psi_\mu$ is its associated multiplicative functional, then we define
\begin{equation}
\label{eq:def-tilde-amu}
\widetilde{a}_\mu(z)\coloneqq
\psi_\mu\big(
\widetilde{a}_{\bullet}(z)
\big).
\end{equation}
If $(\kappa^\alpha)_\alpha$ is any net in $\mathbb N_0^m$ such that $\lim_{\alpha}\kappa^\alpha=\mu$, then
\begin{equation}
\label{eq:tilde-amu=limaka}
\widetilde{a}_\mu(z)=\lim_{\alpha}\widetilde{a}_{\ka^\alpha}(z).
\end{equation}

\begin{rem}
\label{rem:measurability-tildeamu}
    By construction, for every $\mu\in M(\Tkqr)$, the function $\widetilde{a}_\mu(z)$ is well defined for almost every $z\in\bdB^{\bdk}$. It therefore determines a bounded function outside a null set. This construction alone, however, does not guarantee measurability with respect to $z$.

    Indeed, the topology of $M(\Tkqr)$ requires the use of nets; consequently, $\widetilde{a}_\mu$ cannot in general be treated as an almost-everywhere limit of a sequence of measurable functions.

    For this reason, for a general $a\in L^\infty(\bB^n)$, we explicitly assume measurability of the resulting functions $\widetilde{a}_\mu$ when needed. Under the continuity assumption introduced below, measurability will follow from continuity on the angular variable.
\end{rem}

Suppose that $\widetilde{a}_\mu$ is measurable for every $\mu\in M(\Tkqr)$. Since $\widetilde{a}_{\ka}$ is $\bT^m$-invariant for every fixed $\ka$, the function $\widetilde{a}_\mu$ is $\bT^m$-invariant as well. Moreover, $\widetilde{a}_\mu(z)$ depends only on the $\bdS^{\bdk}$-component of $z$ as explained above. As before, we denote by $\widetilde{a}_\mu|_{\bdS^{\bdk}}$ the function on $\bdS^{\bdk}$, unique up to equality almost everywhere, such that
$\widetilde{a}_\mu|_{\bdS^{\bdk}}(\bdxi)=\widetilde{a}_\mu(\bdr\bdxi)$ for almost every $\bdr\in[0,1)^m$.

A natural question concerns the behavior of these operators as $\ka$ converges to a point in $M(\Tkqr)$.
One expects that, under suitable conditions, the operator $\bdT_{\widetilde{a}_\ka}$ approaches the operator $\bdT_{\widetilde{a}_\mu}$.
This question has a satisfactory answer under the following continuity assumption.

Let $\contTm$ denote the space of all $a\in \LinftyTm$ for which a representative of the associated function $f_a\in L^\infty(\tau(\bB^m)\times
\bdS^{\bdk})$ satisfies
\begin{equation}
       \label{eq:limit-condition-symbol}
            \|f_a(\cdot,\bdxi')-f_a(\cdot,\bdxi)\|_{\infty,\tau(\mathbb B^m)}
            \longrightarrow0,
\end{equation}
as $\bdxi'\to \bdxi$ in $\bdS^{\bdk}$.
Equivalently, $\contTm$ consists of those functions $a$ for which $\bdS^{\bdk}\ni\bdxi\mapsto f_a(\cdot,\bdxi)$ is a continuous family of functions in $L^\infty(\tau(\bB^m))$.

In what follows, for $z\in\bdB^{\bdk}$, we denote by $\widetilde{a}_\bullet(z)$ the function $M(\Tkqr)\ni\mu\mapsto \widetilde{a}_\mu(z)$.

\begin{prop}
\label{prop:continuity-wildeamu}
Let $a\in\contTm$. Then
\[
\big\|
\widetilde{a}_\bullet|_{\bdS^{\bdk}}(\bdxi')
-
\widetilde{a}_\bullet|_{\bdS^{\bdk}}(\bdxi)
\big\|_{\infty,M(\Tkqr)}
\to 0,
\]
as $\bdxi'\to\bdxi$ in $\bdS^{\bdk}$.
In particular, for every $\mu\in M(\Tkqr)$ the function
$\widetilde{a}_\mu|_{\bdS^{\bdk}}$ has a continuous representative and
$\widetilde{a}_\mu$ is measurable.
\end{prop}
\begin{proof}
    Equation \eqref{eq:def-a-tilde-kappa} implies that, for $\ka\in \bN_0^m$,
    \begin{align*}
        \big|\widetilde{a}_{\ka}|_{\bdS^{\bdk}}(\bdxi)-\widetilde{a}_{\ka}|_{\bdS^{\bdk}}(\bdxi')\big|
    &\leq
    C\int_{\tau(\bB^m)}|f_a(\bdr,\bdxi)-f_a(\bdr,\bdxi')|\bdr^{2\bdk+2\ka-\bdone}(1-|\bdr|_2^2)^\lambda d\bdr\\
    &\leq
    C
    \|f_a(\cdot,\bdxi)-f_a(\cdot,\bdxi')\|_{\infty}
    \int_{\tau(\bB^m)}\bdr^{2\bdk+2\ka-\bdone}(1-|\bdr|_2^2)^\lambda d\bdr\\
    &    \leq
    \|f_a(\cdot,\bdxi)-f_a(\cdot,\bdxi')\|_{\infty}.
    \end{align*}
    Here, $C=\frac{2^m\Gamma(n+|\kappa|+\lambda+1)}{\Gamma(\lambda+1)(\bdk+\ka-\bdone)!}$ and we used the fact that
    \[
    \frac{2^m\Gamma(n+|\kappa|+\lambda+1)}{\Gamma(\lambda+1)(\bdk+\ka-\bdone)!}
\int_{\tau({\mathbb B}^m)}
\bdr^{2\bdk+2\ka-\bdone}(1-|\bdr|_2^2)^\lambda d\bdr=1.
    \]
    Passing to the limit along a net $\ka^\alpha\to\mu$ gives
    \[
    |\widetilde{a}_\mu|_{\bdS^{\bdk}}(\bdxi)-\widetilde{a}_\mu|_{\bdS^{\bdk}}(\bdxi')|\leq
    \|f_a(\cdot,\bdxi)-f_a(\cdot,\bdxi')\|_{\infty}.
    \]
    Hence, $\widetilde{a}_\mu|_{\bdS^{\bdk}}$ has a continuous representative on $\bdS^{\bdk}$ and is therefore measurable. Consequently, $\widetilde{a}_\mu$ is measurable. The same estimate yields
    \[
    \big\|
    \widetilde{a}_\bullet|_{\bdS^{\bdk}}(\bdxi')
    -
    \widetilde{a}_\bullet|_{\bdS^{\bdk}}(\bdxi)
    \big\|_{\infty,M(\Tkqr)}
    \leq
    \|f_a(\cdot,\bdxi)-f_a(\cdot,\bdxi')\|_{\infty}.
    \]
    Since $a\in\contTm$, the right-hand side tends to zero as $\bdxi'\to\bdxi$ in $\bdS^{\bdk}$.
\end{proof}

\begin{cor}
\label{cor:uniform-approx-aka-amu}
Let $a\in\contTm$. Then the map
\[
M(\Tkqr)\ni \mu\longmapsto \widetilde{a}_\mu\in L^\infty(\bdB^{\bdk})
\]
is continuous. In particular, if $(\ka^\alpha)_\alpha$ is a net in $\bN_0^m$ converging to $\mu\in M(\Tkqr)$ then
\[
\lim_{\alpha}\|
\widetilde{a}_{\ka^\alpha}-\widetilde{a}_\mu
\|_{L^\infty(\bdB^{\bdk})}
=0.
\]
\end{cor}
\begin{proof}
    Let $\mu\in M(\Tkqr)$ and $\varepsilon>0$. By the preceding proposition, there exists $\delta>0$ such that
    \[
    \bigg\| \widetilde{a}_\bullet|_{\bdS^{\bdk}}(\bdxi)-\widetilde{a}_\bullet|_{\bdS^{\bdk}}(\bdxi')
    \bigg\|_{\infty,M(\Tkqr)}<\frac{\varepsilon}{3},
\]
for every $\bdxi,\bdxi'\in\bdS^{\bdk}$ such that $|\bdxi-\bdxi'|<\delta$.
By compactness, $\bdS^{\bdk}$ admits a finite cover $\{B(\bdxi_\ell,\delta)\cap \bdS^{\bdk}\}_{\ell=1}^L$, where each $B(\bdxi_\ell,\delta)$ is the open ball in $\bC^n$ of radius $\delta$ centered at $\bdxi_\ell\in \bdS^{\bdk}$.

Each function $\widetilde{a}_\bullet|_{\bdS^{\bdk}}(\bdxi_\ell)$ belongs to $\Tkqr$, $\ell=1,\ldots,L$. Hence, consider the open set $V$ in the topology of $M(\Tkqr)$ given by
\[
V=
\bigg\{
\eta\in M(\Tkqr)\colon
\big|\widetilde{a}_\mu|_{\bdS^{\bdk}}(\bdxi_{\ell})-
\widetilde{a}_\eta|_{\bdS^{\bdk}}(\bdxi_\ell)\big|<\varepsilon/3,\ \forall \ell=1,\ldots,L
\bigg\}.
\]

Let $\eta\in V$. For any $z\in\bdB^{\bdk}$ with $z_{(j)}\neq 0$ for all $j=1,\ldots,m$, there exists $\ell$ such that $|\bdxi(z)-\bdxi_\ell|<\delta$ and hence
\begin{align*}
\bigg|\widetilde{a}_\mu(z)
-\widetilde{a}_{\eta}(z)\bigg|
&=
\bigg|\widetilde{a}_\mu|_{\bdS^{\bdk}}\big(\bdxi(z)\big)
-\widetilde{a}_{\eta}|_{\bdS^{\bdk}}\big(\bdxi(z)\big)\bigg|\\
&
\leq
\bigg|\widetilde{a}_\mu|_{\bdS^{\bdk}}\big(\bdxi(z)\big)
-\widetilde{a}_{\mu}|_{\bdS^{\bdk}}\big(\bdxi_\ell\big)\bigg|
+\bigg|\widetilde{a}_\mu|_{\bdS^{\bdk}}\big(\bdxi_\ell\big)
-\widetilde{a}_{\eta}|_{\bdS^{\bdk}}\big(\bdxi_\ell\big)\bigg|\\
&
+
\bigg|\widetilde{a}_{\eta}|_{\bdS^{\bdk}}\big(\bdxi_\ell\big)
-\widetilde{a}_{\eta}|_{\bdS^{\bdk}}\big(\bdxi(z)\big)\bigg|<\varepsilon.\qedhere
\end{align*}
\end{proof}

\begin{cor}
\label{cor:Tmu-to-Tkappaalpha}
Let $a\in\contTm$. Then the map
\[
M(\Tkqr)\ni\mu\longmapsto \bdT_{\widetilde{a}_\mu}\in \cL(\cA^2(\bdB^{\bdk}))
\]
is continuous in the operator norm.
In particular, if $(\ka^\alpha)_\alpha$ is a net in $\bN_0^m$ converging to $\mu\in M(\Tkqr)$, then the net $(\bdT_{\widetilde{a}_{\ka^\alpha}})_\alpha$ converges to
$\bdT_{\widetilde{a}_\mu}$ in operator norm.
\end{cor}
\begin{proof}
This follows from Corollary~\ref{cor:uniform-approx-aka-amu} and the estimate
\[
\|
\bdT_{\widetilde a_\mu}-
\bdT_{\widetilde a_\eta}
\|
\leq
\|\widetilde a_{\mu}-\widetilde a_\eta\|_{L^\infty(\bdB^{\bdk})}.\qedhere
\]
\end{proof}

\begin{ex}
\label{ex:example2-asymptotics}
    Let $n=4$, $\lambda=0$, $m=2$ and $\bdk=(2,2)$ as in Examples \ref{ex:example1-symbol} and \ref{ex:example2-symbol}.

    As in the latter example, consider the function $\psi$ with $f_\psi(\bdr,\bdxi)=r_1^{2\alpha} G_1(\bdxi)
    +r_2^{2\beta} G_2(\bdxi)$. Note that this function belongs to $\contTm$ if and only if both functions $G_j$ are continuous on $S^3$.

    In this case, assuming for simplicity that $\alpha,\beta>0$, by Stirling's approximation and \eqref{eq:example2-formula}, one has the following asymptotic values:
    \[
    \widetilde{\psi}_{\ka}(z)\approx
    \begin{cases}
        G_2(\bdxi),& \ka_1\text{ fixed,}\ \ka_2\to\infty,\\
        G_1(\bdxi),& \ka_2\text{ fixed,}\ \ka_1\to\infty,\\
        (\ka_1/|\ka|)^\alpha G_1(\bdxi)
        +
        (\ka_2/|\ka|)^\beta G_2(\bdxi),&
        \ka_1,\ka_2\to\infty.
    \end{cases}
    \]
Thus, the family $(\bdT_{\widetilde{\psi}_\mu})_{\mu}$ is a continuous family of linear combinations of the operators $\bdT_{G_1}$ and $\bdT_{G_2}$.
\end{ex}

\begin{ex}
\label{ex:example1-asymptotics}
    Consider the special case of the symbols $\phi$ from Example~\ref{ex:example1-symbol}:
    \[
    f_\phi(\bdr,\bdxi)=r_1^{iH_1(\bdxi)}
    r_2^{i H_2(\bdxi)}.
    \]
    In general, these symbols do not belong to $\contTm$, even assuming that the functions $H_j$ are continuous. Indeed, this follows from
    \[
    |r^{ih}-r^{ih'}|=|1-e^{i(h-h')\log r}|,
    \]
    which does not converge uniformly to zero as $h\to h'$.

    However, as before, the function $M(\Tkqr)\ni\mu\mapsto\widetilde{\phi}_\mu$ is still well-defined (though not obviously measurable) by \eqref{eq:def-tilde-amu}.
    Note that, by Stirling's approximation and \eqref{eq:example1-formula}, one has
    \[
    \widetilde{\phi}_\ka(z)\approx
    \begin{cases}
        \displaystyle\frac{\Gamma(\frac{i}{2}H_1(\bdxi)+\ka_1+2)}{\Gamma(\ka_1+2)}\ka_2^{-i/2 H_1(\bdxi)},& \ka_1\text{ fixed},\ \ka_2\to\infty,\\
        \displaystyle\frac{\Gamma(\frac{i}{2}H_2(\bdxi)+\ka_2+2)}{\Gamma(\ka_2+2)}\ka_1^{-i/2 H_2(\bdxi)},& \ka_2\text{ fixed},\ \ka_1\to\infty.
    \end{cases}
    \]
    If both $\ka_1,\ka_2\to\infty$, one gets much more involved expressions.

    Thus, the symbols $\widetilde{\phi}_\ka$ have nontrivial oscillatory behavior and, although well-defined, there seems to be no concrete representation for $\widetilde{\phi}_\mu$ with $\mu\in M(\Tkqr)\backslash\bN_0^2$.
\end{ex}

\section{Operator algebras}
\label{sec:Operator algebras}

In this section, we apply the preceding framework to describe commutative Banach algebras generated by $\bT^m$-invariant Toeplitz operators. More specifically, we study the algebra $\Tkqrph$ introduced in \cite{Quiroga2021}, which generalizes the commutative algebras from earlier works.

We work under the continuity assumptions in the preceding section. However, our results can be restated for more general operators whenever they give rise to continuous families of operators on $M(\Tkqr)$.

\subsection{Commutative Banach algebras}

For each $j\in\{1,\ldots,m\}$, consider the group
\begin{equation}
\label{eq:def-group-U-T-U}
\UkjT\coloneqq
U(k_1)\times\cdots \times U(k_{j-1})\times \mathbb T I_{k_j}\times U(k_{j+1})\times\cdots\times
U(k_m),
\end{equation}
where, as before, $U(d)$ denotes the group of $d\times d$ unitary matrices. Here, $I_{k_j}$ denotes the $k_j\times k_j$ identity matrix.

The group $\bT^m$ is a subgroup of $\UkjT$, while the latter is a subgroup of $U(\bdk)$.
Given $j\in\{1,\ldots,m\}$,
let $\LinftyUkjT$ denote the space of all functions in $L^\infty(\bB^n)$ invariant under $\UkjT$.
A direct verification shows that, for $c\in\LinftyUkjT$, there is an essentially unique function $g_c\in L^\infty(\tau(\bB^m)\times S^{2k_j-1})$ such that
\[
c(z)=g_c\left(\bdr(z),\frac{z_{(j)}}{|z_{(j)}|}\right),
\qquad z\in \bB^n\setminus\{z_{(j)}=0\}.
\]
Since $\bT^m\subset\UkjT$,
\[
\LinftyUkjT\subset\LinftyTm.
\]
Note that the associated function $f_c$ defined by \eqref{eq:def-associated-fa-Tm-invariant} satisfies
\[
f_c(\bdr,\bdxi)=g_c(\bdr,\xi_j),
\]
for almost every $\bdr\in\tau(\bB^m)$ and $\bdxi=(\xi_1,\ldots,\xi_m)\in \bdS^{\bdk}$.

For each $j\in\{1,\ldots,m\}$, we define the space
\[
\begin{aligned}
\LinftyjT
={}&\bigg\{f\in L^\infty(\bB^{k_j})\colon\\
&\qquad
f(t^{-1}z)=f(z)\text{ for a.e. }z\in \bB^{k_j}
\text{ for every }t\in \bT
\bigg\}.
\end{aligned}
\]

\begin{prop}
\label{prop:hat-cka-UkjT}
    Let $j\in\{1,\ldots,m\}$ and $c_j\in \LinftyUkjT$. Then the operator $T_{c_j}$ is unitarily equivalent to the direct sum
    \[
    \bigoplus_{\ka\in\bN_0^m}I_{H_{k_1,\ka_1}}\otimes\cdots\otimes T_{k_j,\widehat{c}_{j,\ka}}|_{H_{k_j,\ka_j}}\otimes\cdots\otimes I_{H_{k_m,\ka_m}},
    \]
    where $\widehat{c}_{j,\ka}\in \LinftyjT$ is defined almost everywhere by
    \begin{align*}
    \widehat{c}_{j,\ka}(z_{(j)})
    &={}
    \frac{2^m\Gamma(n+|\ka|+\lambda+1)}{\Gamma(\lambda+1)(\bdk+\ka-\bdone)!}
    \int_{\tau({\mathbb B}^m)}
    g_{c_j}\left(\bdr,\frac{z_{(j)}}{|z_{(j)}|}\right)
    \bdr^{2\bdk+2\ka-\bdone}(1-|\bdr|_2^2)^\lambda d\bdr.
    \end{align*}
    Moreover, for every $\ka\in\bN_0^m$, the following properties hold:
    \begin{enumerate}
        \item $\|\widehat{c}_{j,\ka}\|_\infty\leq \|c_j\|_\infty$.
        \item The function $\widehat{c}_{j,\ka}$ depends only on the angular variable $z_{(j)}/|z_{(j)}|\in S^{2k_j-1}$ of $z_{(j)}\in \bB^{k_j}$.
    \end{enumerate}
\end{prop}
\begin{proof}
    We know from Theorem~\ref{thm:Tna=oplus-T-polyball} that $UT_{n,c_j}U^*=\bigoplus_{\ka\in\bN_0^m}\bdT_{\widetilde{c_j}_\ka}|_{\bdH_{\bdk,\ka}}$, where $\widetilde{c_j}_\ka$ is given by \eqref{eq:def-a-tilde-kappa}.
    In this case, the $\bdxi(z)$-dependence of $\widetilde{c_j}_\ka(z)$ is given only through the $j$th component of $\bdxi(z)$. Thus,
    \[
    \widetilde{c_j}_\ka(z)=
    \widehat{c}_{j,\ka}(z_{(j)}),
    \]
    and, under the natural tensor product structure of $\cA^2(\bdB^{\bdk})$,
    \[
    \bdT_{\widetilde{c_j}_\ka}
    =
    I\otimes\cdots\otimes T_{k_j,\widehat{c}_{j,\ka}}\otimes\cdots\otimes I.
    \]
    Restricting to $\bdH_{\bdk,\ka}\cong H_{k_1,\ka_1}\otimes\cdots\otimes H_{k_m,\ka_m}$ yields the result.
\end{proof}

Whenever $c_j\in\contTm$, for every $\mu\in M(\Tkqr)$ we may define the function $\widehat{c}_{j,\mu}$ by
\begin{equation}
\label{eq:def-cmuj-UkjT}
\widehat{c}_{j,\mu}(z)
=
\psi_\mu\big(\widehat{c}_{j,\bullet}(z)\big),\qquad \text{for almost every }z\in\bB^{k_j}.
\end{equation}
Since $\widehat{c}_{j,\ka}$ is the one-variable factor of $\widetilde{c_j}_\ka$ as a tensor product of functions, the results from Section~\ref{sec:cont-families} imply that $\widehat{c}_{j,\mu}$ is bounded and measurable, with
\[
\|\widehat{c}_{j,\mu}\|_\infty\leq \|c_j\|_\infty.
\]
Moreover, its boundary value $\widehat{c}_{j,\mu}|_{S^{2k_j-1}}$ is continuous on $S^{2k_j-1}$; indeed, by Corollary~\ref{cor:uniform-approx-aka-amu}, it is the uniform limit on the sphere of the boundary values of the functions $\widehat{c}_{j,\ka}$.

\begin{cor}
\label{cor:Tkqrph-is-commutative}
    For every $j\in\{1,\ldots,m\}$ fix a symbol $c_j\in\LinftyUkjT$. Then the Banach algebra
    \[
    \Tkqrph=\cB(\Tkqr,T_{c_1},\ldots,T_{c_m})
    \]
    generated by $\Tkqr$ and the operators $T_{c_1},\ldots,T_{c_m}$ is commutative.
\end{cor}
\begin{proof}
Under the unitary equivalence of Theorem~\ref{thm:Tna=oplus-T-polyball}, every element of $\Tkqr$ is diagonal with respect to the decomposition
$\cA^2_\lambda(\bB^n)=\bigoplus_{\ka\in\bN_0^m}H_{n,\ka}$,
and it acts as a scalar on each block $H_{n,\ka}$. Hence every element of $\Tkqr$ commutes with every $T_{c_j}$.

By Proposition~\ref{prop:hat-cka-UkjT}, on each block
\[
H_{n,\ka}\cong H_{k_1,\ka_1}\otimes\cdots\otimes H_{k_m,\ka_m},
\]
the operator $T_{c_j}$ acts as the identity on all tensor factors except the $j$-th one, where it acts as $T_{k_j,\widehat c_{j,\ka}}|_{H_{k_j,\ka_j}}$. Therefore, if $i\neq j$, the restrictions of $T_{c_i}$ and $T_{c_j}$ act on different tensor factors and commute on every block $H_{n,\ka}$. It follows that all generators commute, and so the Banach algebra generated by them is commutative.
\end{proof}

\begin{ex}
    Consider again the setting $n=4$, $\lambda=0$, $m=2$ and $\bdk=(2,2)$.
    \begin{itemize}
        \item Let $\phi$ be the function from Example \ref{ex:example1-symbol} where the functions $H_j\in C(S^3\times S^3)$ depend only on the first component $\xi_1$ of $\bdxi$. Then $\phi\in L^\infty(\bB^4)^{U(\bdk,1,\bT)}$. This function does not belong to $\contTm$ unless both $H_j$ are constant (see Example \ref{ex:example1-asymptotics}).

        \item Similarly, let $\psi$ be as in Example \ref{ex:example2-symbol}, where both $G_j\in C(S^3\times S^3)$ depend only on the first component of $\bdxi$. Then $\psi\in L^\infty(\bB^4)^{U(\bdk,1,\bT)}$. If $\psi$ is continuous, then one has $\psi\in \contTm\cap L^\infty(\bB^4)^{U(\bdk,1,\bT)}$.
        
    \end{itemize}
    Analogous examples are obtained for the case $j=2$.

    In general, the corresponding Banach algebras $\Tkqrph$ do not belong to any of the classes covered by previous works \cite{RodriguezRodriguezPHD2024,Bauer_Rodriguez2022,GarciaVasilevski2015,BauerVasilevski2015,BauerVasilevski2013}.
\end{ex}

\begin{rem}
\label{rem:quasi-homogeneous-history}
The symbols in $\LinftyUkjT$, and hence the associated Banach algebras $\Tkqrph$, generalize several constructions from earlier works.

    Early examples include the quasi-homogeneous symbols introduced
in \cite{Vasilevski2010}. In the case $m=1$, one fixes
$h\in\{1,\ldots,n-1\}$ and multi-indices $p,q\in\bN_0^n$ of the form
\[
p=(p_1,\ldots,p_h,0,\ldots,0),
\qquad
q=(0,\ldots,0,q_{h+1},\ldots,q_n),
\qquad |p|=|q|.
\]
The associated quasi-homogeneous symbol is
\begin{equation}
\label{eq:def-quasi-homogeneous}
c(z)=
\left(\frac{z}{|z|}\right)^p
\overline{
\left(\frac{z}{|z|}\right)^q},
\qquad z\in\bB^n.
\end{equation}
These are functions in $\contTm\cap\LinftyUkjT$.

For a general partition $\bdk=(k_1,\ldots,k_m)$, the construction is performed
on each group of variables. Thus one obtains symbols of the form
\[
c(z)=
\prod_{j=1}^m
\left(\frac{z_{(j)}}{|z_{(j)}|}\right)^{p_{(j)}}
\overline{
\left(\frac{z_{(j)}}{|z_{(j)}|}\right)^{q_{(j)}}},
\]
where the multi-indices $p_{(j)},q_{(j)}\in\bN_0^{k_j}$ satisfy the analogous
orthogonality and balance conditions in each block. The resulting commutative Banach algebras are of a slightly different type than the algebra $\Tkqrph$ considered in this work, as they admit more generators.
However, their Gelfand theory follows easily from our framework after minor modifications (see \cite{RodriguezRodriguezPHD2024,Bauer_Rodriguez2022} and the references therein).

A substantially broader class of symbols was introduced in \cite{RodriguezVasilevski2021}. They were called \emph{(II)-pseudo-homogeneous} in \cite{RodriguezRodriguezPHD2024}. In this case, the
restriction to functions of the special form above was removed. For each block
$j$, one considers bounded functions depending only on the angular variable
$z_{(j)}/|z_{(j)}|$ and invariant under the scalar action of $\bT$. Equivalently,
one takes symbols of the form
\[
c_j(z)=
f_{c_j}\left(\frac{z_{(j)}}{|z_{(j)}|}\right),
\]
where $f_{c_j}\in L^\infty(S^{2k_j-1})$ satisfies
\[
f_{c_j}(t\xi)=f_{c_j}(\xi),
\qquad t\in\bT,\ \xi\in\partial\bB^{k_j}.
\]
The corresponding Gelfand theory was studied in detail in
\cite{RodriguezVasilevski2021,RodriguezRodriguezPHD2024}.
We remark that, although the class $\LinftyUkjT$ is much larger, allowing quasi-radial dependence of the symbols, the induced symbols $\widehat{c}_{j,\mu}$ are (II)-pseudo-homogeneous.

In \cite{Quiroga2021}, the author referred to the symbols in $\LinftyUkjT$ as
quasi-homogeneous. To avoid confusion with the earlier terminology, we do not give these symbols a separate name. 
Yet, we keep the notation from previous works $\Tkqrph$ to refer to the
Banach algebras generated by $\Tkqr$ and Toeplitz operators with such symbols (where ``ph'' stands for ``pseudo-homogeneous'').
\end{rem}

The remainder of the paper describes the Gelfand theory of the algebra $\Tkqrph$ (under the angular continuity assumptions). As special cases, our results recover the descriptions of the commutative Banach algebras considered in earlier works (\cite{BauerVasilevski2012,BauerVasilevski2013,BauerVasilevski2015,Bauer_Rodriguez2022,RodriguezVasilevski2021,RodriguezRodriguezPHD2024}).

\subsection{\texorpdfstring{On $\bT$-invariant Toeplitz operators with continuous symbols and Calkin algebras}{On T-invariant Toeplitz operators with continuous symbols and Calkin algebras}}
\label{sec:Tinvariant-algebras-continuous}

Before proceeding with our analysis, it will be useful to study operator algebras generated by $\bT$-invariant Toeplitz operators acting on a Bergman space $\cA^2(\bB^d)$.
Here $d\in \bN$ denotes a general dimension that will be specialized later.

Let $\varphi\in L^\infty(\bB^d)$. We say that $\varphi$ admits a continuous boundary value on $\partial\bB^d=S^{2d-1}$ if there exist a measurable representative $\varphi_0$ of $\varphi$ and a function $h\in C(S^{2d-1})$ such that
\begin{equation}
    \label{eq:extension-to-boundary}
    \|
    \varphi_0(r\cdot)-h\|_{\infty,S^{2d-1}}\to0,\quad r\to 1^{-}.
\end{equation}
In this case, the function $h$ is unique and will be denoted by $\varphi|_{S^{2d-1}}$. Whenever radial restrictions are used, we fix such a representative and continue to denote it by $\varphi$.
\begin{lem}
\label{lem:Tvarphi-Tvarphi'-compact}
    Let $\varphi\in L^\infty(\bB^d)$ and suppose that it satisfies \eqref{eq:extension-to-boundary}. If $\varphi'$ is any other function with continuous boundary values such that $\varphi|_{S^{2d-1}}=\varphi'|_{S^{2d-1}}$, then
    \[
    T_{d,\varphi}-T_{d,\varphi'}\in\cK(\cA^2(\bB^d)).
    \]
\end{lem}
\begin{proof}
    Set $h=\varphi-\varphi'$ and, for $r\in[0,1)$, let
    $h_r=h\chi_{r\bB^d}$. Since $h_r$ is bounded and supported in the compact set $r\overline{\bB^d}$, the Toeplitz operator $T_{d,h_r}$ is compact. Indeed, its integral kernel is square-integrable because
    $\int_{r\bB^d}K_d(w,w)\,dv_d(w)<\infty$, where $K_d$ is the Bergman kernel of $\cA^2(\bB^d)$.
    Moreover, one has
    \[
    \|T_{d,h}-T_{d,h_r}\|
    \leq \|h-h_r\|_\infty
    =\operatorname*{ess\,sup}_{|z|>r}|h(z)|
    \longrightarrow0.
    \]
    Thus $T_{d,h}$ is a norm limit of compact operators and is therefore compact.
\end{proof}

\begin{prop}
\label{prop:Calkin-continuous-boundary}
    Let $\varphi\in L^\infty(\bB^d)$ and suppose that it satisfies \eqref{eq:extension-to-boundary}. Let $p(z,\overline{z})$ be a polynomial. Then
    \begin{enumerate}
        \item
        One has
        \[
        p(T_{d,\varphi},T_{d,\varphi}^*)-T_{d,p(\varphi,\overline{\varphi})}\in \cK(\cA^2(\bB^d)).
        \]
        \item The unital $C^*$-algebra $C^*([T_{d,\varphi}]_\cK)$ is isometrically isomorphic to the $C^*$-algebra $C(\varphi|_{S^{2d-1}}(S^{2d-1}))$.
        In particular, the essential spectrum of $T_{d,\varphi}$ is
        \[
        \operatorname{ess-sp}\big(T_{d,\varphi}\big)
        =
        \varphi|_{S^{2d-1}}(S^{2d-1}).
        \]
        \item
        The unital Banach algebra $\cB\big([T_{d,\varphi}]_{\cK}\big)$
        generated by $[T_{d,\varphi}]_{\cK}$ inside
        $\cL(\cA^2(\bB^d))/\cK(\cA^2(\bB^d))$ is isomorphic to the unital Banach algebra $\cP(\varphi|_{S^{2d-1}}(S^{2d-1}))$ inside $C\big(\varphi|_{S^{2d-1}}(S^{2d-1})\big)$ generated by all holomorphic polynomials.

        In particular, the maximal ideal space of the algebra $\cB\big([T_{d,\varphi}]_{\cK}\big)$ is homeomorphic to the space
        \[
        M\big(\cB\big([T_{d,\varphi}]_{\cK}\big)\big)
        \cong
        \pconv{\operatorname{ess-sp}(T_{d,\varphi})}
        =
        \pconv{\varphi|_{S^{2d-1}}(S^{2d-1})}.
        \]
    \end{enumerate}
\end{prop}
\begin{proof}
    Let $h=\varphi|_{S^{2d-1}}$ and choose an extension
    $\Phi\in C(\overline{\bB^d})$ with $\Phi|_{S^{2d-1}}=h$.
    Lemma~\ref{lem:Tvarphi-Tvarphi'-compact} gives
    $[T_{d,\varphi}]_\cK=[T_{d,\Phi}]_\cK$.

    We use the standard short exact sequence
    \[
    0\longrightarrow \cK(\cA^2(\bB^d))
    \longrightarrow \cT(C(\overline{\bB^d}))
    \overset{\sigma}{\longrightarrow} C(S^{2d-1})
    \longrightarrow0,
    \qquad
    \sigma(T_{d,f})=f|_{S^{2d-1}},
    \]
    where $\cT(C(\overline{\bB^d}))$ is the $C^*$-algebra generated by the Toeplitz operators with continuous symbols; see, for example,
    \cite{Venugopalkrishna1972,Coburn1973,Vasilevski2008,BauerVasilevski2013}.
    Hence the image of $[T_{d,\varphi}]_\cK$ under the symbol map is $h$.

    For a polynomial $p(z,\overline z)$, multiplicativity of $\sigma$ gives
    \[
    p([T_{d,\Phi}]_\cK,[T_{d,\Phi}]_\cK^*)
    =[T_{d,p(\Phi,\overline\Phi)}]_\cK.
    \]
    The symbols $p(\varphi,\overline\varphi)$ and
    $p(\Phi,\overline\Phi)$ have the same continuous boundary value
    $p(h,\overline h)$, so another application of
    Lemma~\ref{lem:Tvarphi-Tvarphi'-compact} proves the first assertion.

    The unital $C^*$-algebra generated by $h$ in $C(S^{2d-1})$ is naturally isometrically isomorphic to $C(h(S^{2d-1}))$. This proves the second assertion and the formula for the essential spectrum. Likewise, the norm closure of the holomorphic polynomials in $h$ is naturally isometrically isomorphic to
    $\cP(h(S^{2d-1}))$, which proves the third assertion and the stated description of its maximal ideal space.
\end{proof}
Consider now the action of $\bT$ on functions $f$ defined on $\bB^d$ by
\[
t\cdot f(z)=f(t^{-1}z),\qquad z\in \bB^d,\ t\in\bT.
\]
Restricting the action to $\cA^2(\bB^d)$ yields a unitary representation $\pi_{\bT}$. An operator $T\in \cL(\cA^2(\bB^d))$ is called $\bT$-invariant if, for every $t\in\bT$,
\[
\pi_{\bT}(t)T=T\pi_{\bT}(t).
\]
For every such operator, the same argument used in Section~\ref{sec:Tm-invariant-Toep-op} yields the direct-sum decomposition $T=\bigoplus_{\ell\in \bN_0}T|_{H_{d,\ell}}$.

Let
\[
\LinftydT=\{\varphi\in L^\infty(\bB^d)\colon
\varphi(t^{-1}z)=\varphi(z)\text{ for a.e. }z\in\bB^d
\text{ for every }t\in\bT\},
\]
and let $\contdT$ denote the space of all $c\in \LinftydT$ that admit a continuous boundary value on $\partial\bB^d=S^{2d-1}$ in the sense that some representative of $c$ satisfies \eqref{eq:extension-to-boundary}.
As in the preceding cases, a Toeplitz operator $T_c$ with symbol $c\in L^\infty(\bB^d)$ commutes with $\pi_{\bT}$ if and only if $c\in\LinftydT$.

Let $\ell\in\bN_0$ and let
\[
K_{d,\ell}(z,w)
=
\sum_{|\alpha|=\ell}e_{d,\alpha}^{(0)}(z)\overline{e_{d,\alpha}^{(0)}(w)},\qquad z,w\in \bB^d,
\]
denote the reproducing kernel of the space $H_{d,\ell}\subset \cA^2(\bB^d)$. Hence, by the multinomial theorem, one obtains
\begin{equation}
    \label{eq:Kdell-multinomial-formula}
    K_{d,\ell}(z,w)=\frac{(d+\ell)!}{d!\ell!}(z\cdot \overline{w})^{\ell},
\end{equation}
where $z\cdot \overline{w}\coloneqq z_1\overline{w_1}+\cdots+z_d\overline{w_d}$. In particular,
\begin{equation}
\label{eq:Kdell-squared-norm}
    \|K_{d,\ell}(\cdot,w)\|^2=K_{d,\ell}(w,w)=\frac{(d+\ell)!}{d!\ell!}|w|^{2\ell}.
\end{equation}
Let $k_{d,\ell}$ denote the normalized reproducing kernel
\[
k_{d,\ell}(z,w)=\frac{K_{d,\ell}(z,w)}{\|K_{d,\ell}(\cdot,w)\|},
\qquad z\in\bB^d,\quad w\in\bB^d\setminus\{0\}.
\]
\begin{prop}
\label{prop:Berezin-approximation}
Let $c\in \contdT$. Then, for every $w\in\bB^d\setminus\{0\}$,
\begin{equation*}
    \lim_{\ell\to\infty}
    \bigl\|\bigl(T_{d,c}-c(w/|w|)I\bigr)k_{d,\ell}(\cdot,w)\bigr\|=0.
\end{equation*}
\end{prop}
\begin{proof}
    We adapt the proof from \cite[Proposition 3.2.5]{RodriguezRodriguezPHD2024} for completeness.
    Since $c\in\contdT$, the function $c|_{S^{2d-1}}$ is uniformly continuous.
    Let $c_{\text{rad}}\in L^\infty(\bB^d)$ be the radial extension of $c|_{S^{2d-1}}$.
    By Lemma~\ref{lem:Tvarphi-Tvarphi'-compact}, the operator $T_{d,c}-T_{d,c_{\text{rad}}}$ is compact. Since $(k_{d,\ell}(\cdot,w))_\ell$ converges weakly to zero, the limit remains the same under compact perturbations of $T_{d,c}$. Thus, we may assume that $c=c_{\text{rad}}$.

    Denote $\xi_w=w/|w|\in S^{2d-1}$.
    Let $\varepsilon>0$ and choose $0<\delta<1$ such that $|\xi-\xi_w|<\delta$ implies $\bigl|c|_{S^{2d-1}}(\xi)-c|_{S^{2d-1}}(\xi_w)\bigr|<\varepsilon$. Define
    \[
    O_\delta=
    \left\{
    z\in\bB^d\setminus\{0\}\colon \left|\frac{z}{|z|}-t\xi_w\right|<\delta,\text{ for some }t\in\bT
    \right\}.
    \]
    \begin{enumerate}
        \item If $z\in O_\delta$, then, by the invariance of $c$, for some $t\in\bT$ we have
    \[
    |c(z)-c(w)|=\bigl|c|_{S^{2d-1}}(z/|z|)-c|_{S^{2d-1}}(t\xi_w)\bigr|<\varepsilon.
    \]

    \item If $z\notin O_\delta$, then for every $t\in\bT$ one has
    $\left|\frac{z}{|z|}-t\xi_w\right|\geq\delta$.
    Squaring both sides and simplifying, we obtain, by
    taking $t=z\cdot\overline{w}/|z\cdot\overline{w}|$ when
    $z\cdot\overline{w}\neq0$,
    \[
    |z\cdot\overline{w}|\leq |z||w|\left(1-\frac{\delta^2}{2}\right).
    \]
    \end{enumerate}
    The estimate is immediate if $z\cdot\overline{w}=0$.
    Therefore, using \eqref{eq:Kdell-multinomial-formula}, \eqref{eq:Kdell-squared-norm} and \eqref{eq:int-z-alpha-conj-z-beta-dvlambda}, we have
    \begin{align*}
        \|&(T_{d,c}-c(w/|w|)I)k_{d,\ell}(\cdot,w)\|^2
        \leq
        \int_{\bB^d}|c(z)-c(w)|^2
        |k_{d,\ell}(z,w)|^2\,dv_d(z)\\
        &\leq
        \varepsilon^2+4\|c\|_\infty^2
        \frac{d!\ell!}{(d+\ell)!|w|^{2\ell}}
        \int_{\bB^d\setminus O_\delta}
        \left|
        \frac{(d+\ell)!}{\ell!d!}(z\cdot\overline{w})^\ell
        \right|^{2}dv_d(z)\\
        &\leq
        \varepsilon^2+4\|c\|_\infty^2
        \frac{(d+\ell)!}{d!\ell!}
        (1-\delta^2/2)^{2\ell}
        \int_{\bB^d\setminus O_\delta}|z|^{2\ell}dv_d(z).
    \end{align*}
    Using \eqref{eq:measures-polar-coordinates}, one can check that $\int_{\bB^d}|z|^{2\ell}dv_{d}(z)=\frac{d}{d+\ell}$. Thus
    \[
    \|(T_{d,c}-c(w/|w|)I)k_{d,\ell}(\cdot,w)\|^2
    \leq
    \varepsilon^2+4\|c\|_\infty^2 (1-\delta^2/2)^{2\ell} \frac{(d+\ell-1)!}{(d-1)!\ell!},
    \]
    where $(1-\delta^2/2)^{2\ell} \frac{(d+\ell-1)!}{(d-1)!\ell!}\to0$ as $\ell\to\infty$.
\end{proof}

\begin{prop}
\label{prop:norm[T]K=lim-norms-T|H}
    Let $c\in \contdT$. Then
    \[
    \big\|[T_{d,c}]_{\cK}\big\|
    =
    \lim_{\ell\to\infty}
    \big\|T_{d,c}|_{H_{d,\ell}}\big\|.
    \]
    Moreover, for every polynomial $p(z,\overline{z})$,
    \[
    \|p([T_{d,c}]_\cK,[T_{d,c}]^*_\cK)\|=\lim_{\ell\to\infty}\|p(T_{d,c},T_{d,c}^*)|_{H_{d,\ell}}\|.
    \]
\end{prop}
\begin{proof}
Proposition~\ref{prop:Calkin-continuous-boundary} yields
\[
\|[T_{d,c}]_{\cK}\|=\|c|_{S^{2d-1}}\|_{\infty}=\bigl|c|_{S^{2d-1}}(\xi)\bigr|,
\]
for some $\xi\in S^{2d-1}$. Let $w\in \bB^d\setminus\{0\}$ with $w/|w|=\xi$.
Proposition~\ref{prop:Berezin-approximation} gives
\[
\|[T_{d,c}]_{\cK}\|=
\lim_{\ell\to\infty}\|T_{d,c}k_{d,\ell}(\cdot,w)\|
\leq
\liminf_{\ell\to\infty}\|T_{d,c}|_{H_{d,\ell}}\|.
\]
On the other hand, let $\varepsilon>0$. Then there exists $K\in \cK(\cA^2(\bB^d))$ such that
\begin{align*}
\|[T_{d,c}]_{\cK}\|&
\geq
\|T_{d,c}+K\|-\varepsilon/2\\
&\geq
\|(T_{d,c}+K)|_{H_{d,\ell}}\|-\varepsilon/2\\
&\geq
\bigg|
\|T_{d,c}|_{H_{d,\ell}}\|-\|K|_{H_{d,\ell}}\|
\bigg|-\varepsilon/2,
\end{align*}
for all $\ell\in\bN_0$.

Since $K$ is compact, one has $\lim_{\ell\to\infty}\|K|_{H_{d,\ell}}\|=0$. Otherwise, one could build a sequence $(f_{\ell_{p}})_p$ with $f_{\ell_{p}}\in H_{d,\ell_{p}}$, $\|f_{\ell_p}\|=1$ and $\|Kf_{\ell_{p}}\|>\delta$ for some $\delta>0$ and all $p$. This is impossible because $(f_{\ell_{p}})$ converges weakly to zero.

Hence, there exists $L\in\bN_0$ such that $\ell\geq L$ implies $\|K|_{H_{d,\ell}}\|<\varepsilon/2$. Thus
\[
\|[T_{d,c}]_{\cK}\|\geq
\|T_{d,c}|_{H_{d,\ell}}\|-\varepsilon
\]
and
\[
\liminf_{\ell\to\infty}
\|T_{d,c}|_{H_{d,\ell}}\|
\geq \|[T_{d,c}]_{\cK}\|
\geq
\limsup_{\ell\to\infty}
\|T_{d,c}|_{H_{d,\ell}}\|-\varepsilon,
\]
and the first limit follows.

For the second assertion, note that
\[
p(T_{d,c},T_{d,c}^*)-T_{d,p(c,\overline{c})}\in \cK(\cA^2(\bB^d)).
\]
Hence
\[
\lim_{\ell\to\infty}
\|(p(T_{d,c},T_{d,c}^*)-T_{d,p(c,\overline{c})})|_{H_{d,\ell}}\|=0
\]
and, by the first limit,
\begin{align*}
    \lim_{\ell\to\infty}\|p(T_{d,c},&T_{d,c}^*)|_{H_{d,\ell}}\|=
\lim_{\ell\to\infty}
\|T_{d,p(c,\overline{c})}|_{H_{d,\ell}}\|\\
&=
\|[T_{d,p(c,\overline{c})}]_\cK\|=
\|p([T_{d,c}]_\cK,[T_{d,c}]_\cK^*)\|.\qedhere
\end{align*}

\end{proof}

\subsection{Local quotient algebras for continuous families of operators}

Let $j\in\{1,\ldots,m\}$ and let $c_j\in\contTm\cap \LinftyUkjT$. Moreover, fix $\mu\in M(\Tkqr)$. By Corollary~\ref{cor:Tmu-to-Tkappaalpha} and Proposition~\ref{prop:hat-cka-UkjT}, when $\ka$ is close to $\mu$ in $M(\Tkqr)$,
\[
T_{n,c_j}|_{H_{n,\ka}}
\approx
I|_{H_{k_1,\ka_1}}\otimes\cdots\otimes T_{k_j,\widehat{c}_{j,\mu}}|_{H_{k_j,\ka_j}}\otimes\cdots\otimes I|_{H_{k_m,\ka_m}}.
\]
Thus, as $\ka\to\mu$, the algebra generated by
$T_{n,c_j}|_{H_{n,\ka}}$ is naturally modeled by the algebra generated by
$T_{k_j,\widehat{c}_{j,\mu}}|_{H_{k_j,\ka_j}}$.
The next two theorems make this heuristic precise.

The result takes two forms, depending on whether the coordinate $\omega_j$ of the stratum containing $\mu$ is finite or infinite. We state the two cases separately.

We formalize this idea of considering limit algebras by constructing suitable quotient spaces.
We let
\[
\Tkqrcj=\cB(\Tkqr,T_{n,c_j})
\]
denote the Banach algebra generated by $\Tkqr$ and the single Toeplitz operator $T_{n,c_j}$.
Let $\Dkqrcj$ denote the dense, nonclosed subalgebra of $\Tkqrcj$ consisting of finite linear combinations of finite products of the generators of $\Tkqrcj$:
\[
\Dkqrcj=
\operatorname{span}
\big\{D_\gamma T_{n,c_j}^{p}\colon D_\gamma\in\Tkqr,\ p\in\bN_0
\big\}.
\]
Given $\mu\in M(\Tkqr)$, let $\psi_\mu$ denote the corresponding multiplicative functional, and let $\ker_\mu=\ker(\psi_\mu)$ be the associated maximal ideal of $\Tkqr$.
Let $\Jmuj$ denote the closed ideal generated inside $\Tkqrcj$ by $\ker_\mu$. Explicitly,
\[
\Jmuj=\overline{\operatorname{span}}
\bigg\{
D_\gamma T_{n,c_j}^p\colon p\in\bN_0,\ \gamma(\mu)=0
\bigg\}.
\]
We denote the equivalence classes in $\Tkqrcj/\Jmuj$ by $[T]_\cJ$.

\begin{lem}
\label{lem:restricted-T-in-Jmu-to-zero}
    Let $(\ka^\alpha)_\alpha$ be a net in $\bN_0^m$ converging to $\mu$ in $M(\Tkqr)$. Then, for every $T\in\Jmuj$,
    \[
    \lim_{\alpha}\|T|_{H_{n,\ka^\alpha}}\|=0.
    \]
\end{lem}
\begin{proof}
    For every $\varepsilon>0$, there exists an operator
    \[
    T_\varepsilon=\sum_{p=0}^P D_{\gamma_p}T_{n,c_j}^p,
    \]
    where $\gamma_p(\mu)=\lim_{\alpha}\gamma_p(\ka^\alpha)=0$ for
    $p=0,\ldots,P$, such that
    \[
    \|T-T_\varepsilon\|<\varepsilon/2.
    \]
    Note that
    \[
    T_\varepsilon|_{H_{n,\ka^\alpha}}
    =
    \sum_{p=0}^P \gamma_p(\ka^\alpha)T^p_{n,c_j}|_{H_{n,\ka^\alpha}}.
    \]
    Thus, there exists $\alpha_0$ such that $\alpha\succeq \alpha_0$ implies $\|T_\varepsilon|_{H_{n,\ka^\alpha}}\|<\varepsilon/2$ and therefore
    \[
    \|T|_{H_{n,\ka^\alpha}}\|
    \leq
    \|(T-T_\varepsilon)|_{H_{n,\ka^\alpha}}\|+\|T_\varepsilon|_{H_{n,\ka^\alpha}}\|<\varepsilon.\qedhere
    \]
\end{proof}

\begin{prop}
\label{prop:norm-TJ-is-limit-along-net}
    Let $T\in \Tkqrcj$. Then there exists a net $(\ka^\alpha)_\alpha$ depending on $T$ and converging to $\mu\in M(\Tkqr)$ such that
    \[
    \|[T]_{\cJ}\|= \lim_{\alpha}\|T|_{H_{n,\ka^\alpha}}\|.
    \]
\end{prop}
\begin{proof}
Let $T\in\Tkqrcj$. We may suppose that $\|[T]_\cJ\|\neq0$.

If $\mu=\ka\in\bN_0^m$, then $[T]_\cJ=[P_{\ka}T]_{\cJ}$, where $P_{\ka}$ is the orthogonal projection from $\cA^2_\lambda(\bB^n)$ onto $H_{n,\ka}$. Thus $\|[T]_\cJ\|\leq \|T|_{H_{n,\ka}}\|$. Furthermore, if $S\in \Jmuj$, then $P_\ka S|_{H_{n,\ka}}=0$. Hence,
\[
\|T|_{H_{n,\ka}}\|\leq
\|(T+S)|_{H_{n,\ka}}\|+\|S|_{H_{n,\ka}}\|
\leq
\|T+S\|.
\]
Taking the infimum over $S\in\Jmuj$, we get $\|[T]_\cJ\|\leq \|T|_{H_{n,\ka}}\|\leq \|[T]_\cJ\|$.

Suppose now that $\mu\in M(\Tkqr)\setminus\bN_0^m$.
    Let $V_\mu$ denote the set of open neighborhoods of $\mu$ in $M(\Tkqr)$.

    By the density of $\bN_0^m$ in $M(\Tkqr)$, for every $U\in V_\mu$ there exists $\eta^U\in U\cap \bN_0^m$. This yields a family $(\eta^U)_{U\in V_\mu}$ of elements in $\bN_0^m$, where $\eta^U\in U$ for all open $U\in V_\mu$. The set $V_\mu$ is a directed set under the usual relation
    \[
    U_1 \preceq U_2\quad\iff\quad U_1\supset U_2.
    \]
    Hence, by construction, $\lim_{U\in V_\mu}\eta^U=\mu$. Suppose that $\mu\in M(\omega)$. Since $\mu\notin \bN_0^m$, at least one coordinate, say $\omega_i$, equals $\infty$. The restriction map $M(\Tkqr)\to\Omega$ is continuous, so $\eta_i^U\to\infty$ in $\bN_0\cup\{\infty\}$. Consequently, $|\eta^U|\to\infty$ and
    $\lim_{U\in V_\mu}\frac{1}{|\eta^U|+1}=0$.

    For every $U\in V_\mu$, Urysohn's lemma provides a function
    $f_U\in C(M(\Tkqr))$ such that $0\leq f_U\leq1$,
    $f_U|_{M(\Tkqr)\setminus U}\equiv0$, and $f_U(\mu)=1$.

    Under the identification $\Tkqr\cong C(M(\Tkqr))$, the function
    $f_U$ corresponds to a diagonal operator $D_{f_U}\in\Tkqr$. In the quotient,
    \[
    [T]_{\cJ}
    =[D_{f_U}T+(I-D_{f_U})T]_{\cJ}
    =[D_{f_U}T]_{\cJ},
    \]
    and therefore
    \[
    \|[T]_{\cJ}\|\leq \|D_{f_U}T\|
    =
    \sup_{\ka\in \bN_0^m}\|D_{f_U}T|_{H_{n,\ka}}\|
    =
    \sup_{\ka\in \bN_0^m}|f_U(\ka)|\|T|_{H_{n,\ka}}\|,
    \]
    where the direct-sum decomposition of $T$ was used.
    Thus, there exists $\ka^U\in\bN_0^m$ such that
    \begin{equation}
        \label{eq:nets-ineq-1}
    \|[T]_\cJ\|\leq \frac{1}{|\eta^U|+1}+|f_U(\ka^U)|\|T|_{H_{n,\ka^U}}\|
    \leq
    \frac{1}{|\eta^U|+1}+\|T|_{H_{n,\ka^U}}\|.
    \end{equation}
    As $U\in V_\mu$ was arbitrary, this defines a net $(\ka^U)_{U\in V_\mu}$.

    We show that $(\ka^U)$ converges to $\mu$ as well.
    Since $[T]_\cJ\neq0$, there exists $U_0\in V_\mu$ such that $U\succeq U_0$ implies $\frac{1}{|\eta^U|+1}<\|[T]_\cJ\|/2$.
    Therefore, $U\succeq U_0$ entails
    \[
    |f_U(\ka^U)|\|T|_{H_{n,\ka^U}}\| \geq \|[T]_\cJ\|/2>0.
    \]
    By the construction of $f_U$, this implies that $\ka^U\in U$. Hence, $\lim_{U\in V_\mu}\ka^U=\mu$.

    Finally, let $\varepsilon>0$ and $S\in \Jmuj$ be such that $\|T+S\|\leq \|[T]_\cJ\|+\varepsilon/2$.
    Then, for every $U\in V_\mu$,
    \begin{align*}
    \|T|_{H_{n,\ka^U}}\|
    &\leq
    \|(T+S)|_{H_{n,\ka^U}}\|+\|S|_{H_{n,\ka^U}}\|\\
    &\leq
    \|T+S\|+\|S|_{H_{n,\ka^U}}\|\\
    &\leq
    \|[T]_\cJ\|+\|S|_{H_{n,\ka^U}}\|+\varepsilon/2.
    \end{align*}
    By Lemma~\ref{lem:restricted-T-in-Jmu-to-zero}, the net $(\|S|_{H_{n,\ka^U}}\|)_U$ converges to zero, hence there exists $U_1$ such that $U\succeq U_1$ implies $\|S|_{H_{n,\ka^U}}\|<\varepsilon/2$ and thus
    \[
    \|T|_{H_{n,\ka^U}}\|<
    \|[T]_\cJ\|+\varepsilon.
    \]
    Moreover, by \eqref{eq:nets-ineq-1} there exists $U_2$ such that $U\succeq U_2$ implies
    \[
    \|[T]_\cJ\|< \varepsilon+\|T|_{H_{n,\ka^U}}\|.
    \]
    Hence, choosing $U_3\succeq U_1,U_2$, we find for every $U\succeq U_3$
    \[
    \big|
    \|[T]_\cJ\|- \|T|_{H_{n,\ka^U}}\|
    \big|
    < \varepsilon.\qedhere
    \]
\end{proof}

\begin{thm}
\label{thm:local-alg-finite}
    Let $c_j\in \contTm\cap\LinftyUkjT$.
    Assume that $\mu \in M(\omega)$ with $\omega_j\in\bN_0$. Then there exists a unique isometric isomorphism of Banach algebras
    \[
    \Psi\colon \cB\big(
    T_{k_j,\widehat{c}_{j,\mu}}|_{H_{k_j,\omega_j}}
    \big)\longrightarrow
    \Tkqrcj/\Jmuj
    \]
    such that
    \[
    \Psi\big(T_{k_j,\widehat{c}_{j,\mu}}|_{H_{k_j,\omega_j}}\big)
    =
    [T_{n,c_j}]_{\cJ}.
    \]
    Here, $\cB\big(
    T_{k_j,\widehat{c}_{j,\mu}}|_{H_{k_j,\omega_j}}
    \big)$ is the unital Banach algebra generated by $T_{k_j,\widehat{c}_{j,\mu}}|_{H_{k_j,\omega_j}}$ and the restricted identity $I|_{H_{k_j,\omega_j}}$.
\end{thm}
\begin{proof}
    Let $\cD$ be the space of operators of the form $p(T_{k_j,\widehat{c}_{j,\mu}}|_{H_{k_j,\omega_j}})$, with $p$ a holomorphic polynomial. This space is dense in $\cB\big(
    T_{k_j,\widehat{c}_{j,\mu}}|_{H_{k_j,\omega_j}}
    \big)$.
    Define the map
    \[
    \widetilde{\Psi}\colon \cD\to \Tkqrcj/\Jmuj
    \]
    by
    \[
    \widetilde{\Psi}(p(T_{k_j,\widehat{c}_{j,\mu}}|_{H_{k_j,\omega_j}}))
    =
    p([T_{n,c_j}]_\cJ).
    \]
    We show that this map is well-defined and isometric. Let $p(z)$ be a holomorphic polynomial and set $S_p=p(T_{n,c_j})$. Recall the orthogonal projection $Q_{j,\omega_j}$ defined by \eqref{eq:def-Qjd}. Since $\psi_\mu(Q_{j,\omega_j})=1$, we have
    \[
    [S_p]_\cJ=[Q_{j,\omega_j}S_p]_\cJ.
    \]
    By Proposition~\ref{prop:norm-TJ-is-limit-along-net}, there is a net $(\ka^\alpha)_\alpha$ in $\bN_0^m$ converging to $\mu$ such that
    \[
    \|[S_p]_\cJ\|=\lim_\alpha\|Q_{j,\omega_j}S_p|_{H_{n,\ka^\alpha}}\|.
    \]
    The Gelfand transform of the projection $Q_{j,\omega_j}$ is a continuous $\{0,1\}$-valued function. Since $\psi_\mu(Q_{j,\omega_j})=1$, it equals $1$ on a neighborhood of $\mu$. Therefore, $Q_{j,\omega_j}(\ka^\alpha)=1$ eventually; equivalently, $\ka_j^\alpha=\omega_j$ eventually. Hence, using Proposition~\ref{prop:hat-cka-UkjT}, we obtain
    \[
    \|[S_p]_\cJ\|
    =
    \lim_\alpha
    \big\|p(T_{k_j,\widehat{c}_{j,\ka^\alpha}})|_{H_{k_j,\omega_j}}\big\|.
    \]
    By the uniform convergence $\|\widehat{c}_{j,\ka^\alpha}-\widehat{c}_{j,\mu}\|_\infty\to0$, polynomial functional calculus for bounded operators yields
    \[
    \lim_\alpha
    \big\|p(T_{k_j,\widehat{c}_{j,\ka^\alpha}})|_{H_{k_j,\omega_j}}
    -p(T_{k_j,\widehat{c}_{j,\mu}})|_{H_{k_j,\omega_j}}\big\|=0.
    \]
    Therefore
    \[
    \|p([T_{n,c_j}]_\cJ)\|
    =
    \big\|p(T_{k_j,\widehat{c}_{j,\mu}})|_{H_{k_j,\omega_j}}\big\|.
    \]

    Thus $\widetilde{\Psi}$ is an isometric homomorphism and extends to an isometric homomorphism $\Psi\colon\cB\big(
    T_{k_j,\widehat{c}_{j,\mu}}|_{H_{k_j,\omega_j}}
    \big)\to
    \Tkqrcj/\Jmuj$. It is surjective because the quotient is generated by $[T_{n,c_j}]_\cJ$ and the identity: indeed, $[D_\gamma]_\cJ=\gamma(\mu)[I]_\cJ$ for every $D_\gamma\in\Tkqr$. Being defined through the generators, the isomorphism is unique.
\end{proof}

\begin{thm}
\label{thm:local-alg-infinite}
Let $c_j\in \contTm\cap\LinftyUkjT$.
Assume that $\mu\in M(\omega)$ and that $\omega_j=\infty$. Then there exists a unique isometric isomorphism of Banach algebras
\[
\Phi\colon
\cB\big(\big[T_{k_j,\widehat{c}_{j,\mu}}\big]_{\cK}\big)
\longrightarrow
\Tkqrcj/\Jmuj
\]
such that
\[
\Phi([T_{k_j,\widehat{c}_{j,\mu}}]_\cK)=[T_{n,c_j}]_\cJ.
\]
Here, $\cB\big(\big[T_{k_j,\widehat{c}_{j,\mu}}\big]_{\cK}\big)$ is the unital Banach algebra
generated by $[T_{k_j,\widehat{c}_{j,\mu}}]_{\cK}$ inside the Calkin algebra $\cL(\cA^2(\bB^{k_j}))/\cK(\cA^2(\bB^{k_j}))$ and $\widehat{c}_{j,\mu}$ is given by \eqref{eq:def-cmuj-UkjT}.
\end{thm}
\begin{proof}
    Let $\cD_\cK$ be the set of all operators of the form $p([T_{k_j,\widehat{c}_{j,\mu}}]_{\cK})$, with $p$ a holomorphic polynomial. The space $\cD_\cK$ is dense in
$\cB([T_{k_j,\widehat{c}_{j,\mu}}]_{\cK})$.

    By construction, $\widehat{c}_{j,\mu}$ is $\bT$-invariant, is constant along radial directions, and has a continuous boundary value; hence $\widehat{c}_{j,\mu}\in\contdT$.

    Let $p$ be a holomorphic polynomial and set $S_p=p(T_{n,c_j})$. By Proposition~\ref{prop:norm-TJ-is-limit-along-net}, there is a net $(\ka^\alpha)_\alpha$ converging to $\mu$ such that
    \[
        \|p([T_{n,c_j}]_\cJ)\|=\lim_\alpha\|S_p|_{H_{n,\ka^\alpha}}\|.
    \]
    Since $\omega_j=\infty$, we have $\ka_j^\alpha\to\infty$. Proposition~\ref{prop:hat-cka-UkjT} yields
    \[
    \|p([T_{n,c_j}]_\cJ)\|
    =
    \lim_\alpha\big\|p(T_{k_j,\widehat{c}_{j,\ka^\alpha}})|_{H_{k_j,\ka_j^\alpha}}\big\|.
    \]
    The uniform convergence $\|\widehat{c}_{j,\ka^\alpha}-\widehat{c}_{j,\mu}\|_\infty\to0$ implies, by continuity of polynomial functional calculus in the operator norm, that
    \[
    \lim_\alpha
    \big\|
    p(T_{k_j,\widehat{c}_{j,\ka^\alpha}})|_{H_{k_j,\ka_j^\alpha}}
    -p(T_{k_j,\widehat{c}_{j,\mu}})|_{H_{k_j,\ka_j^\alpha}}
    \big\|=0.
    \]
    Hence
    \[
    \|p([T_{n,c_j}]_\cJ)\|
    =
    \lim_\alpha
    \big\|p(T_{k_j,\widehat{c}_{j,\mu}})|_{H_{k_j,\ka_j^\alpha}}\big\|.
    \]
    Since $\ka_j^\alpha\to\infty$, Proposition~\ref{prop:norm[T]K=lim-norms-T|H} yields
    \[
    \lim_\alpha
    \big\|p(T_{k_j,\widehat{c}_{j,\mu}})|_{H_{k_j,\ka_j^\alpha}}\big\|
    =
    \|p([T_{k_j,\widehat{c}_{j,\mu}}]_{\cK})\|,
    \]
    and therefore
    \[
    \|p([T_{n,c_j}]_\cJ)\|=\|p([T_{k_j,\widehat{c}_{j,\mu}}]_{\cK})\|
    \]
    for every holomorphic polynomial $p$.

    Define the map $\widetilde{\Phi}\colon \cD_\cK\to \Tkqrcj/\Jmuj$ by
    \[
    \widetilde{\Phi}(p([T_{k_j,\widehat{c}_{j,\mu}}]_\cK))
    =
    p(T_{n,c_j})+\Jmuj.
    \]
    The preceding equality shows that $\widetilde{\Phi}$ is well-defined and isometric. Its definition on polynomials also shows that it is a homomorphism.

    It therefore extends to an isometric homomorphism
    \[
    \Phi\colon \cB([T_{k_j,\widehat{c}_{j,\mu}}]_{\cK})\to
    \Tkqrcj/\Jmuj.
    \]
    It is surjective because $\Phi([I]_\cK)=[I]_\cJ$, where $[D_\gamma]_\cJ=\gamma(\mu)[I]_{\cJ}$ for every $D_\gamma\in \Tkqr$, and $\Phi([T_{k_j,\widehat{c}_{j,\mu}}]_\cK)=[T_{n,c_j}]_\cJ$. Because it is determined by the generators, the isomorphism is unique.
\end{proof}

\begin{cor}
\label{cor:maximal-ideal-space-limit-algebras}
    One has
    \[
        M(\Tkqrcj/\Jmuj)\cong
        \begin{cases}
        \pconv{\widehat{c}_{j,\mu}|_{S^{2k_j-1}}(S^{2k_j-1})},&\quad \omega_j=\infty,\\
        \operatorname{sp}(T_{k_j,\widehat{c}_{j,\mu}}|_{H_{k_j,\omega_j}}),&\quad \omega_j\in\bN_0.
        \end{cases}
        \]
\end{cor}
\begin{proof}
    The infinite-coordinate case follows from Theorem~\ref{thm:local-alg-infinite} and Proposition~\ref{prop:Calkin-continuous-boundary}. In the finite-coordinate case, Theorem~\ref{thm:local-alg-finite} reduces the quotient to the unital Banach algebra generated by an operator on the finite-dimensional space $H_{k_j,\omega_j}$. Its spectrum is finite and therefore polynomially convex, so the maximal ideal space is the displayed spectrum.
\end{proof}

\subsection{\texorpdfstring{Gelfand theory of $\Tkqrph$}{Gelfand theory of Tkqrph}}

For each $j\in\{1,\ldots,m\}$ fix a symbol $c_j\in \LinftyUkjT\cap \contTm$.
Let
\[
\Tkqrph
=
\cB(\Tkqr,T_{c_1},\ldots,T_{c_m})
\]
be the Banach algebra generated by $\Tkqr$ and the Toeplitz operators $T_{c_1},\ldots,T_{c_m}$.
Let $M(\Tkqrph)$ denote its compact space of maximal ideals. Let $\Dkqrph$ denote the dense, nonclosed subalgebra of $\Tkqrph$ consisting of finite linear combinations of finite products of the generators of $\Tkqrph$:
\[
\Dkqrph=\operatorname{span}\{D_\gamma T^\rho_c,\ D_\gamma\in\Tkqr,\ \rho\in\bN_0^m\},
\]
where
\[
T^\rho_c=T_{n,c_1}^{\rho_1}\cdots T_{n,c_m}^{\rho_m}.
\]

Note that the algebra $\Tkqrph$ is generated by the $C^*$-algebra $\Tkqr$ and the unital Banach algebra $\Tph$ generated by the operators $T_{c_1},\ldots,T_{c_m}$. Thus, as in earlier work, the space $M(\Tkqrph)$ is continuously embedded in the Cartesian product of the maximal ideal spaces of the two algebras:
\[
M(\Tkqrph)\hookrightarrow
M(\Tkqr)\times M(\Tph)
\]
through the map
\[
M(\Tkqrph)\ni \psi \mapsto (\psi|_{\Tkqr},\psi|_{\Tph}).
\]
Similarly, because the algebra $\Tph$ is generated by the operators $T_{c_1},\ldots,T_{c_m}$ (and the identity operator $I$), we have
\[
M(\Tph)\subset
M(\mathcal{B}(T_{c_1}))\times\cdots\times
M(\mathcal{B}(T_{c_m})),
\]
where $\mathcal{B}(T_{c_j})$ denotes the unital Banach algebra generated by $T_{c_j}$. For any operator $T$, the maximal ideal space of the unital Banach algebra generated by $T$ is homeomorphic to the polynomially convex hull of $\operatorname{sp}(T)$.
Consequently, we obtain
\[
M(\Tph)\subset \pconv{\operatorname{sp}(T_{c_1})}\times\cdots\times\pconv{\operatorname{sp}(T_{c_m})}\subset \bC^m.
\]

To describe the Gelfand theory of $\Tkqrph$, we introduce some notation.
Let $\mu\in M(\Tkqr)$ and suppose that $\mu\in M(\omega)$, according to \eqref{eq:M=bicup-Momega}. We define the sets
\begin{equation}
\label{eq:def-D-mu-cj}
D_j(\mu,c_j)=
\begin{cases}
\operatorname{sp}(T_{k_j,\widehat{c}_{j,\mu}}|_{H_{k_j,d}}),&\quad\textup{if}\quad\omega_j=d\in\bN_0,\\
\widehat{c}_{j,\mu}|_{S^{2k_j-1}}(S^{2k_j-1}),&\quad\textup{if}\quad\omega_j=\infty.
\end{cases}
\end{equation}
Here, $\widehat{c}_{j,\mu}|_{S^{2k_j-1}}$ denotes the continuous boundary value of $\widehat{c}_{j,\mu}$. Since $\widehat{c}_{j,\mu}$ is constant along radial directions, a representative can be chosen so that
\[
\widehat{c}_{j,\mu}(r\xi)
=\widehat{c}_{j,\mu}|_{S^{2k_j-1}}(\xi)
\]
for almost every $(r,\xi)\in [0,1)\times S^{2k_j-1}$.

\begin{prop}
\label{prop:cup-subset-M(Tkqrph)}
    One has
    \[
    \bigcup_{\mu\in M(\Tkqr)}\{\mu\}\times\pconv{D_1(\mu,c_1)}\times\cdots\times \pconv{D_m(\mu,c_m)}
    \subset M(\Tkqrph).
    \]
\end{prop}
\begin{proof}
    Denote $D(\mu,c)=D_1(\mu,c_1)\times\cdots\times D_m(\mu,c_m)$.
    We first prove that $\{\mu\}\times D(\mu,c)\subset M(\Tkqrph)$.

    Let $(\mu,\zeta)\in\{\mu\}\times D(\mu,c)$ and let $(\ka^\alpha)_\alpha$ be a net in $\bN_0^m$ converging to $\mu$. After passing to a cofinal tail if necessary, we may assume that $\ka_j^\alpha=\omega_j$ whenever $\omega_j\in\bN_0$, while $\ka_j^\alpha\to\infty$ whenever $\omega_j=\infty$.

    \begin{itemize}
        \item 
    For each $j$ with $\omega_j=d\in\bN_0$, choose a unit eigenvector $v_j\in H_{k_j,d}$ such that
    \[
    T_{k_j,\widehat{c}_{j,\mu}}v_j=\zeta_jv_j.
    \]
    Such a vector exists because $H_{k_j,d}$ is finite-dimensional and
    \[
    \zeta_j\in
    \operatorname{sp}\big(T_{k_j,\widehat{c}_{j,\mu}}|_{H_{k_j,d}}\big).
    \]
    For these indices, set $g_j^\alpha=v_j$.

    \item 
    For each $j$ with $\omega_j=\infty$, choose $\xi_j\in S^{2k_j-1}$ such that
    \[
    \widehat{c}_{j,\mu}|_{S^{2k_j-1}}(\xi_j)=\zeta_j,
    \]
    and choose $w_j\in\bB^{k_j}\setminus\{0\}$ with $w_j/|w_j|=\xi_j$. We then define
    \[
    g_j^\alpha=k_{k_j,\ka_j^\alpha}(\cdot,w_j)\in H_{k_j,\ka_j^\alpha},
    \]
    where $k_{k_j,\ka_j^\alpha}$ is the normalized reproducing kernel introduced in Proposition~\ref{prop:Berezin-approximation}. Then each $g_j^\alpha$ is a unit vector.
    \end{itemize}

    Define
    \[
    g^\alpha=g_1^\alpha\otimes\cdots\otimes g_m^\alpha\in\bdH_{\bdk,\ka^\alpha},
    \qquad
    h^\alpha=U^*g^\alpha\in H_{n,\ka^\alpha}.
    \]
    Let $T=\sum_{\rho\in F}D_{\gamma_\rho}T_c^\rho\in\Dkqrph$, where $F\subset\bN_0^m$ is finite. Since $D_{\gamma_\rho}$ acts as the scalar $\gamma_\rho(\ka^\alpha)$ on $H_{n,\ka^\alpha}$, Proposition~\ref{prop:hat-cka-UkjT} yields
    \[
    T_c^\rho h^\alpha
    =
    U^*\big(
    T_{k_1,\widehat{c}_{1,\ka^\alpha}}^{\rho_1}g_1^\alpha\otimes\cdots\otimes
    T_{k_m,\widehat{c}_{m,\ka^\alpha}}^{\rho_m}g_m^\alpha
    \big).
    \]
    The uniform convergence of the symbols implies $\|T_{k_j,\widehat{c}_{j,\ka^\alpha}}-T_{k_j,\widehat{c}_{j,\mu}}\|\to0$, while Proposition~\ref{prop:Berezin-approximation} gives
    $\|(T_{k_j,\widehat{c}_{j,\mu}}-\zeta_j I)g_j^\alpha\|\to0$, whenever $\omega_j=\infty$. It follows that
    \[
    \lim_\alpha
    \langle T_{k_j,\widehat{c}_{j,\ka^\alpha}}^{\rho_j}g_j^\alpha,g_j^\alpha\rangle
    =\zeta_j^{\rho_j},\quad j=1,\ldots,m.
    \]
    Consequently,
    \[
    \lim_\alpha\langle Th^\alpha,h^\alpha\rangle
    =
    \sum_{\rho\in F}\gamma_\rho(\mu)\zeta^\rho.
    \]
    Now consider the map on $\Dkqrph$ given by
    \[
    \psi_{(\mu,\zeta)}\left(\sum_{\rho\in F}D_{\gamma_\rho}T_c^\rho\right)
    =
    \sum_{\rho\in F}\gamma_\rho(\mu)\zeta^\rho.
    \]
    The preceding limit shows that this is a well-defined bounded linear functional on $\Dkqrph$.
    To verify multiplicativity, let $D_{\gamma_1},D_{\gamma_2}\in\Tkqr$ and $\rho_1,\rho_2\in\bN_0^m$. Then
    \begin{align*}
    \psi_{(\mu,\zeta)}
    ((D_{\gamma_1}T^{\rho_1}_c)(D_{\gamma_2}T^{\rho_2}_c))
    &=
    \psi_{(\mu,\zeta)}
    (
    D_{\gamma_1\gamma_2}T^{\rho_1+\rho_2}_c
    )\\
    &=
    (\gamma_1\gamma_2)(\mu)\zeta^{\rho_1+\rho_2}\\
    &=
    \psi_{(\mu,\zeta)}(D_{\gamma_1}T_c^{\rho_1})\psi_{(\mu,\zeta)}(D_{\gamma_2}T_c^{\rho_2}),
    \end{align*}
    where addition of multi-indices is componentwise.
    Thus $\psi_{(\mu,\zeta)}$ is multiplicative and therefore extends to a multiplicative functional on $\Tkqrph$. This proves
    \[
    \{\mu\}\times D(\mu,c)\subset M(\Tkqrph).
    \]
    Finally, let $\zeta\in\pconv{D(\mu,c)}$. If $T=\sum_{\rho\in F}D_{\gamma_\rho}T_c^\rho\in \Dkqrph$, then the polynomial $z\mapsto\sum_{\rho\in F}\gamma_\rho(\mu)z^\rho$ satisfies
    \begin{align*}
    \left|\sum_{\rho\in F}\gamma_\rho(\mu)\zeta^\rho\right|
    &\leq \sup_{z\in D(\mu,c)}
    \left|\sum_{\rho\in F}\gamma_\rho(\mu)z^\rho\right|=
    \sup_{z\in D(\mu,c)}|\psi_{(\mu,z)}(T)|
    \leq
    \|T\|.
    \end{align*}
    Thus $(\mu,\zeta)$ defines a bounded multiplicative functional on $\Dkqrph$ that extends continuously to $\Tkqrph$. Therefore,
    \[
    \{\mu\}\times\pconv{D(\mu,c)}\subset M(\Tkqrph),
    \]
    which is the desired inclusion.
\end{proof}

We conclude by characterizing the Gelfand spectrum and Gelfand transform of the algebra $\Tkqrph$. The inclusion established in Proposition~\ref{prop:cup-subset-M(Tkqrph)} was proved in \cite[Section 3]{RodriguezRodriguezPHD2024}, and the reverse inclusion was conjectured there. The following theorem proves that this conjecture is correct. It also resolves the continuous-symbol part of Open Problem~4 posed at the end of
\cite{Bauer_Rodriguez2022}; the part concerning symbols without continuity
assumptions remains open here.
\begin{thm}
\label{thm:Gelfand theory Tkqrph}
The maximal ideal space of the Banach algebra $\Tkqrph$ is homeomorphic to the subset of $M(\Tkqr)\times\bC^m$ given by
    \begin{equation}
    \label{eq:cup-maximal-ideal-space-Tkqrph}
    \bigcup_{\mu\in M(\Tkqr)}\{\mu\}\times\pconv{D_1(\mu,c_1)}\times\cdots\times \pconv{D_m(\mu,c_m)}.
    \end{equation}
    The homeomorphism is given by
    \[
    M(\Tkqrph)\ni\psi\mapsto (\psi|_{\Tkqr},\psi(T_{n,c_1}),\ldots,\psi(T_{n,c_m})).
    \]
    The Gelfand transform
    \[
    \Gamma\colon \Tkqrph\longrightarrow C(M(\Tkqrph))
    \]
    is determined on the dense subalgebra $\Dkqrph$:
    \[
    \Gamma\left(
    \sum_{\rho\in F}D_{\gamma_\rho}T_c^\rho
    \right)(\mu,\zeta)
    =
    \sum_{\rho\in F}\gamma_\rho(\mu)\zeta^\rho,\quad
    (\mu,\zeta)\in M(\Tkqrph),
    \]
    where $F$ is a finite subset of $\bN_0^m$, $D_{\gamma_\rho}\in\Tkqr$ for all $\rho\in F$ and $T_c^\rho=T_{n,c_1}^{\rho_1}\cdots T_{n,c_m}^{\rho_m}$.
\end{thm}
\begin{proof}
    By the preceding proposition, it remains only to prove that
    \[
    M(\Tkqrph)\subset
    \bigcup_{\mu\in M(\Tkqr)}\{\mu\}\times\pconv{D_1(\mu,c_1)}\times\cdots\times \pconv{D_m(\mu,c_m)}.
    \]

    Let $\psi\in M(\Tkqrph)$, and write $(\mu,\zeta)\in M(\Tkqr)\times\bC^m$ for its image under the natural embedding described above. Denote the corresponding character by $\psi_{(\mu,\zeta)}$.

    Let $j\in\{1,\ldots,m\}$. The restriction $\psi_{(\mu,\zeta)}|_{\Tkqrcj}$ defines a multiplicative functional on the Banach algebra $\Tkqrcj$. Moreover,
    \[
    \psi_{(\mu,\zeta)}(D_\gamma T^\ell_{n,c_j})=\gamma(\mu)\zeta_j^\ell=0,
    \]
    for every $D_\gamma\in\Tkqr$ with $\gamma(\mu)=0$. Hence, $\psi_{(\mu,\zeta)}|_{\Tkqrcj}$ descends to the quotient $\Tkqrcj/\Jmuj$ and therefore defines a multiplicative functional $\widetilde{\psi}_{(\mu,\zeta)}$ on the algebra $\Tkqrcj/\Jmuj$.

    Corollary~\ref{cor:maximal-ideal-space-limit-algebras} now gives
    \[
    \zeta_j=
    \psi_{(\mu,\zeta)}(T_{n,c_j})
    =
    \widetilde{\psi}_{(\mu,\zeta)}([T_{n,c_j}]_\cJ)
    \in \pconv{D_j(\mu,c_j)}.
    \]
Therefore, the map
\[
M(\Tkqrph)\ni \psi\mapsto
(\psi|_{\Tkqr},\psi(T_{n,c_1}),\ldots,\psi(T_{n,c_m}))\in
M(\Tkqr)\times\bC^m
\]
has image equal to the set \eqref{eq:cup-maximal-ideal-space-Tkqrph}. The map is injective because $\Tkqrph$ is generated by $\Tkqr$ and $T_{n,c_1},\ldots,T_{n,c_m}$. Because $M(\Tkqrph)$ is compact and $M(\Tkqr)\times\bC^m$ is Hausdorff, it is therefore a homeomorphism onto its image.
The formula for the Gelfand transform follows by applying these characters in $M(\Tkqrph)$ to the operators in $\Dkqrph$.
\end{proof}

Applying the preceding results to the special cases in Remark~\ref{rem:quasi-homogeneous-history} recovers the corresponding conclusions from earlier works (see \cite{Bauer_Rodriguez2022,BauerVasilevski2012,BauerVasilevski2013,BauerVasilevski2015,RodriguezRodriguezPHD2024}).

\subsection{A final example}
\label{ex:explicit-fibers}

Consider again the setting $n=4$, $\lambda=0$, $m=2$, and
$\bdk=(2,2)$. Define the continuous $\bT$-invariant functions
$P,Q\colon S^3\to\bC$ by
\[
P(w_1,w_2)=\re(w_1\overline{w_2}),
\qquad
Q(w_1,w_2)=w_1\overline{w_2}.
\]
We regard these functions as radially constant symbols on $\bB^2$.
The symbol $Q$ is a standard quasi-homogeneous symbol of degree $1$ studied in
\cite{BauerVasilevski2012} (see also Remark~\ref{rem:quasi-homogeneous-history}), while $P$ is a particular case of the symbols studied in \cite[Section 3.7]{RodriguezRodriguezPHD2024}.
Notice that
\[
P(S^3)=[-1/2,1/2],
\qquad
Q(S^3)=\{z\in\bC:|z|\leq 1/2\}.
\]
Following the recipe from Example~\ref{ex:example2-symbol}, fix $\alpha,\beta>0$ and define symbols $c_1,c_2$ by
\[
f_{c_1}(\bdr,\bdxi)
=
r_1^{2\alpha}Q(\xi_1)
+
r_2^{2\beta}P(\xi_1),
\]
and
\[
f_{c_2}(\bdr,\bdxi)
=
r_1^{2\alpha}P(\xi_2)
+
r_2^{2\beta}Q(\xi_2),
\]
where
\[
\bdr=(r_1,r_2)\in\tau(\bB^2),
\qquad
\bdxi=(\xi_1,\xi_2)\in S^3\times S^3.
\]
Since $P$ and $Q$ are continuous, we have
$c_j\in\contTm\cap L^\infty(\bB^4)^{U(\bdk,j,\bT)}$ for $j=1,2$.

For $\ka=(\ka_1,\ka_2)\in\bN_0^2$, set
\[
a_\ka
=
\frac{B(|\ka|+5,\alpha)}
     {B(\ka_1+2,\alpha)},
\qquad
b_\ka
=
\frac{B(|\ka|+5,\beta)}
     {B(\ka_2+2,\beta)},
\]
where $B(\cdot,\cdot)$ denotes the Beta function, as in
Example~\ref{ex:example2-symbol}. By
\eqref{eq:example2-formula}, we obtain
\[
\widehat c_{1,\ka}
=
a_\ka Q+b_\ka P,
\qquad
\widehat c_{2,\ka}
=
a_\ka P+b_\ka Q.
\]
Since the sequences $a_\bullet$ and $b_\bullet$ belong to $\Tkqr$,
for every $\mu\in M(\Tkqr)$ we may define
\[
a_\mu=\psi_\mu(a_\bullet),
\qquad
b_\mu=\psi_\mu(b_\bullet).
\]
Therefore,
\begin{equation}
\label{eq:final-example-local-symbols}
\widehat c_{1,\mu}=a_\mu Q+b_\mu P,
\qquad
\widehat c_{2,\mu}=a_\mu P+b_\mu Q.
\end{equation}
For $\mu\in M(\Tkqr)$, denote the fiber of $M(\Tkqrph)$ over $\mu$, according to Theorem~\ref{thm:Gelfand theory Tkqrph},
by
\[
\mathfrak F_\mu
=
\pconv{D_1(\mu,c_1)}
\times
\pconv{D_2(\mu,c_2)}.
\]
We now describe how $\mathfrak F_\mu$ changes according to the
stratum of $M(\Tkqr)$ containing $\mu$.

\medskip

\noindent
\textbf{1. Finite coordinates.}
Let $\mu=(d_1,d_2)\in\bN_0^2$.
One has
\[
a_\mu=a_{(d_1,d_2)},
\qquad
b_\mu=b_{(d_1,d_2)}.
\]
Hence
\[
\mathfrak F_\mu
=
\operatorname{sp}
\left(
T_{2,a_\mu Q+b_\mu P}\big|_{H_{2,d_1}}
\right)
\times
\operatorname{sp}
\left(
T_{2,a_\mu P+b_\mu Q}\big|_{H_{2,d_2}}
\right).
\]
Both factors are finite sets. Thus, the fibers over
$\bN_0^2\subset M(\Tkqr)$ are finite.

\medskip

\noindent
\textbf{2. The stratum $M((d,\infty))$.}
Suppose that
$\mu\in M((d,\infty))$.
If $\ka^\nu\to\mu$, then $\ka_1^\nu=d$ eventually and
$\ka_2^\nu\to\infty$. Using the asymptotics of the Beta function (see Example~\ref{ex:example2-asymptotics}),
\[
a_\mu=0,
\qquad
b_\mu=1.
\]
Consequently,
\[
\widehat c_{1,\mu}=P,
\qquad
\widehat c_{2,\mu}=Q.
\]
The first coordinate of the stratum is finite, whereas the second is
infinite. Therefore,
\[
\mathfrak F_\mu
=
\operatorname{sp}
\left(
T_{2,P}\big|_{H_{2,d}}
\right)
\times
\{z\in\bC:|z|\leq 1/2\}.
\]

\medskip

\noindent
\textbf{3. The stratum $M((\infty,d))$.}
Similarly, if $\mu\in M((\infty,d))$, then
\[
\mathfrak F_\mu
=
\{z\in\bC:|z|\leq 1/2\}
\times
\operatorname{sp}
\left(
T_{2,P}\big|_{H_{2,d}}
\right).
\]

\medskip

\noindent
\textbf{4. The stratum $M((\infty,\infty))$.}
Finally, let
$\mu\in M((\infty,\infty))$ and put
\[
a=a_\mu,
\qquad
b=b_\mu.
\]
By \eqref{eq:final-example-local-symbols},
\[
\widehat c_{1,\mu}=aQ+bP,
\qquad
\widehat c_{2,\mu}=aP+bQ.
\]
In this case, the fibers may vary even within the single stratum
$M((\infty,\infty))$. To make the parameter explicit, consider the sequence
\[
s_\ka=\frac{\ka_1+2}{|\ka|+5}
=\gamma_{r_1^2,0}(\ka),
\]
which belongs to $\Tkqr$, and set $\theta=\psi_\mu(s_\bullet)\in[0,1]$.
Since $s_\ka-\ka_1/|\ka|\to0$ as $|\ka|\to\infty$, every net
$\ka^\nu\to\mu$ satisfies
\[
\frac{\ka_1^\nu}{|\ka^\nu|}\longrightarrow\theta.
\]
The asymptotics of the Beta function therefore give
\[
a_\mu=\theta^\alpha,
\qquad
b_\mu=(1-\theta)^\beta.
\]
Thus,
\[
D_1(\mu,c_1)
=
\left\{
\big(
\theta^\alpha+(1-\theta)^\beta
\big)x
+
i\theta^\alpha y:
x^2+y^2\leq\frac14
\right\},
\]
whereas
\[
D_2(\mu,c_2)
=
\left\{
\big(
\theta^\alpha+(1-\theta)^\beta
\big)x
+
i(1-\theta)^\beta y:
x^2+y^2\leq\frac14
\right\}.
\]
Both displayed sets are filled ellipses, possibly degenerating to intervals.
They are convex and hence polynomially convex. Consequently,
\[
\mathfrak F_\mu=D_1(\mu,c_1)\times D_2(\mu,c_2).
\]

\section*{Acknowledgements}

This work was supported by SECIHTI through the program
\emph{Estancias posdoctorales por México}, CVU 860740.

\printbibliography

\section*{Author information}

\noindent\textbf{Miguel Angel Rodriguez Rodriguez}\\
Escuela Superior de Física y Matemáticas, Instituto Politécnico Nacional\\
Ciudad de México, C.P. 07738, Mexico\\
\href{mailto:miarodriguezro@ipn.mx}{miarodriguezro@ipn.mx}\\
\href{https://orcid.org/0000-0002-5124-8271}{ORCID: 0000-0002-5124-8271}

\end{document}